\documentclass[10pt]{article}
 \usepackage[dvips]{graphicx}
 \usepackage{psfrag}
 \let\theoremstyle\relax
 \usepackage{amsfonts, amssymb, amsmath, amsthm}
\usepackage{color}
\usepackage{graphicx, array, blindtext}
 \usepackage{mathabx}
\usepackage{euscript}
\usepackage{mathrsfs}
\usepackage{appendix}
\usepackage{enumitem}
\usepackage{lscape}
\usepackage{multirow}
\usepackage[table,xcdraw]{xcolor}
\usepackage{longtable}
\usepackage{graphicx, array, blindtext}
\usepackage[space]{grffile}
\usepackage{float}
\usepackage{caption}
\usepackage{subfigure}
\numberwithin{equation}{section}
\DeclareMathAlphabet{\mathpzc}{OT1}{pzc}{m}{it}
\large\normalsize

\def\R{\hbox{\bf\rlap{I}{\hbox to 2 pt{}}R}}

\theoremstyle{theorem}
\newtheorem{thm}{Theorem}[section]
\newtheorem{lem}[thm]{Lemma}

\newtheorem{prop}[thm]{Proposition}
\newtheorem{cor}[thm]{Corollary}

\newtheorem{defn}[thm]{Definition}

\theoremstyle{remark}
\newtheorem{rem}[thm]{Remark}

 \title{Global Classical Solutions Near Vacuum to the Initial-Boundary Value Problem of Isentropic Supersonic Flows through Divergent Ducts}
  \author{
  Ying-Chieh Lin\footnote{Department of Applied Mathematics, National University of Kaohsiung, Kaohsiung 81148, Taiwan,
  	  E-mail: linyj@nuk.edu.tw},\ \
  Jay Chu\footnote{Department of Mathematics, National Tsing Hua University, Hsinchu 30013, Taiwan,
      E-mail: ccchu@math.nthu.edu.tw},\ \
  John M. Hong\footnote{Department of Mathematics, National Central University, Chung-Li 32001, Taiwan,
      E-mail: jhong@math.ncu.edu.tw},\ and \
 Hsin-Yi Lee\footnote{Department of Mathematics, National Central University, Chung-Li 32001, Taiwan,
      E-mail: mathlee1209@gmail.com}}

\begin{document}
\date{}
\maketitle


\medskip

\medskip

\begin{abstract}

In this paper, we study the global existence and asymptotic behavior of classical solutions near vacuum for the initial-boundary value problem modeling isentropic supersonic flows through divergent ducts. The governing equations are the compressible Euler equations with a small parameter, which can be written as a hyperbolic system  in terms of the Riemann invariants with a non-dissipative source. We provide a new result for the global existence of classical solutions to initial-boundary value problems of non-dissipative hyperbolic balance laws without the assumption of small data. The work is based on the local existence, the maximum principle, and the uniform a priori estimates obtained by the generalized Lax transformations. The asymptotic behavior of classical solutions is also shown by studying the behavior of Riemann invariants along each characteristic curve and vertical line. The results can be applied to the spherically symmetric solutions to $N$-dimensional compressible Euler equations. Numerical simulations are provided to support our theoretical results.

\vspace{3mm}


\end{abstract}

\section{Introduction}
\setcounter{section}{1}

In this paper, we study the classical solution near vacuum to the initial-boundary value problem for the one-dimensional isentropic supersonic flow through a divergent duct, which can be modeled as the following initial-boundary value problem for compressible Euler equations with a small parameter $0 < \eta \ll 1$ :
\begin{equation}
\label{a.1}
\begin{cases}
(\rho_\eta)_{t} + (\rho_\eta v)_{x} = -\frac{a'(x)}{a(x)} \rho_\eta  v,& x \in \mathring{I}, \ \  t >0, \\
\rho_{\eta}(v_{t} + vv_{x}) + (P(\rho_\eta))_{x} = 0,&  x \in \mathring{I}, \ \  t >0, \\
\rho_\eta (x,0) = \eta \rho_{0}(x), \ \ v(x,0) = v_{0}(x),&  x \in I, \\
\rho_\eta (x_B,t) = \eta \rho_{B}(t), \ \ v(x_B,t) = v_{B}(t),&  t \geq 0,
\end{cases}
\end{equation}
where $I$ is an interval representing either $[x_B, x_C]$ or $[x_B, \infty)$, $\mathring{I}$ denotes the interior of $I$, $\rho_\eta$ is the density of gas, $v$ is the velocity, $a=a(x)\in C^2(I)$ is the cross-sectional area of a duct, and $P$ is the pressure of gas satisfying the $\gamma$-law: $P(\rho)=\rho^{\gamma}, \ \ 1< \gamma <3$. We say the flow is supersonic if $|v|>\sqrt{P'(\rho_\eta)}$. In addition, we assume that $a(x) > 0$ and $a'(x)\ge 0$ for all $x \in I$. Following \cite{Chu}, we define the new parameter
\begin{equation*}
\label{a.4}
\nu := \frac{\gamma}{\gamma - 1} \eta^{\gamma -1}.
\end{equation*}
Then using the rescaling
\begin{equation}
\label{a.5}
\rho_{\eta} (x,t) = \left( \frac{\gamma -1}{\gamma} \nu \right)^{\frac{1}{\gamma-1}} \rho (x,t),
\end{equation}
problem \eqref{a.1} can be reformulated as the following:
\begin{equation}
\label{a.8}
\begin{cases}
\rho_{t} + (\rho v)_{x} =-\frac{a'(x)}{a(x)} \rho v ,& x \in \mathring{I}, \ \  t >0, \\
v_{t} + ( \frac{v^2}{2} + \nu \rho^{\gamma-1} )_{x} = 0,&  x \in \mathring{I}, \ \  t >0, \\
\rho(x,0) = \rho_{0}(x), \ \ v(x,0) = v_{0}(x),&  x \in I, \\
\rho(x_B,t) = \rho_{B}(t), \ \ v(x_B,t) = v_{B}(t),&  t \geq 0.
\end{cases}
\end{equation}
The purpose of the paper is twofold. One is to establish the global existence of the classical solution near vacuum ($0 < \eta \ll 1$) to problem
\eqref{a.1} or, equivalently, the global existence of the classical solution to problem \eqref{a.8}, and apply these results to the initial-boundary value problem of compressible Euler equations in $N$-dimensional spherically symmetric space-times. The other is to obtain the asymptotic behavior of the classical solution along each characteristic curve and vertical line.

We recall some previous results related to the global classical solutions or the singularity formation to hyperbolic systems of conservation or balance laws. The global existence of classical solutions of the Cauchy problem for $2 \times 2$ conservation laws can be seen in the Li's book \cite{TLi1}, or in the earlier paper \cite{Yamaguti} written by Yamaguti and Nishida. The system in \cite{TLi1} (or \cite{Yamaguti}) was written as a canonical form in terms of the Riemann invariants. In \cite{TLi1}, the global existence result was achieved by the local existence of the solution shown in \cite{TLi3} and the uniform a priori estimate for the solution. The transformation introduced by Lax \cite{Lax} was used to transform the equations for the first derivatives of Riemann invariants into the Riccati equations. To the dissipative system
\begin{equation}
\label{a.15}
u_t+\Lambda(x,t,u) u_x+g(x,t,u)=0,
\end{equation}
where $\Lambda(x,t, u)$ is the diagonal $n \times n$ matrix and $g(x,t,u)$ satisfies the dissipative condition, the global existence of classical solutions was established by Li and Qin \cite{TLi2} when $\Lambda(x,t,u)=\Lambda(u)$ and $g(x,t,u)=g(u)$. The result was obtained under the assumption that the $C^1$ norm of the initial data is small. In \cite{Hsiao}, under the smallness of the $C^1$ norm of initial data, Hsiao and Li extended the global existence result to the dissipative system whose $\Lambda(x,t,u) u_x$ is replaced by $f(u)_x$ and $g(x,t,u)=g(u)$. In \cite{Yang2}, Yang, Zhu, and Zhao studied the global classical solution of \eqref{a.15} when the $C^0$ norm of the initial data is small. The global existence result to the damping $p$-system with the small $C^0$ norm of initial data were obtained by Wang and Li \cite{Wang}. In \cite{Jiang}, Jiang, Li, and Ruan considered an initial-boundary value problem with nonlinear damping and the null Dirichlet boundary condition. The global classical solution was obtained if the $C^0$ norm of the first derivative of initial data is small. Ying, Yang, and Zhu in \cite{Ying} studied the axisymmetric and spherically symmetric solutions to the multi-dimensional compressible Euler equations, which can be reformulated as a $2 \times 2$ hyperbolic system of balance laws with damping. They established the global existence of classical solutions when the damping is strong enough. To the systems with relaxation, the global classical solution of the Cauchy problem was obtained by Zhu \cite{Zhu} and Yang and Zhu \cite{Yang1} under the different assumptions of small data. In \cite{Li1}, Li and Liu obtained the critical thresholds of a relaxation model by providing the conditions for the global regularity or finite-time singularity of solutions. The critical thresholds of pressureless Euler-Poisson equations was first given by Engelberg, Liu, and Tadmor \cite{Engelberg}, and recently studied by Bhatnagar and Liu \cite{Bhatnagar}. To the related topics on Euler-Poisson equations, we refer the readers to the papers by Luo and Smoller \cite{Luo1} or Luo, Xin, and Zeng \cite{Luo2}. To the steady states of multi-dimensional compressible Euler equations, we refer the readers to articles by Wang and Xin \cite{Wang2}.

The global classical solutions of compressible Euler equations with general pressure laws in Lagrangian coordinate satisfy the $3 \times 3$ system
\begin{equation}
\label{a.16}
\begin{cases}
v_t-u_x=0,\\
u_t+p(v,s)_x=0,\\
s_t=0,
\end{cases}
\end{equation}
where $s$ is the entropy and $p(v,s)=kv^{-\gamma}\exp(s/c)$ for some positive constant $k$. Since $s = s(x)$, system \eqref{a.16}
can be rewritten as
\begin{equation}
\label{a.17}
\begin{cases}
v_t-u_x=0,\\
u_t+P(v,x)_x=0
\end{cases}
\end{equation}
for some function $P$. The global existence of classical solutions to an initial-boundary value problem of \eqref{a.16} was established by Lin, Liu, and Yang \cite{Lin2} under some restrictions on the initial and boundary data. To the Cauchy problem for system \eqref{a.17} with a more general function $P$, Chen and Young \cite{Chen4} studied the mechanism of a singularity formation and the life span of the classical solutions. Chen, Pan, and Zhu in \cite{Chen1} also studied the singularity formation for compressible Euler equations. These results were extended to some related systems by Chen, Young, and Zhang \cite{Chen5}, such as MHD equations and compressible Euler equations for the fluid in variable area ducts. To the subject of fluid near vacuum (or containing vacuum), we refer the readers to \cite{Chu, Lee, TPLiu4, TPLiu5}. To the weak solutions of hyperbolic conservation laws, we refer the readers to \cite{GChen2, Glimm1, Hong, Lee, TPLiu3, Makino, Temple, Tsuge3}.

The above results give a significant contribution to the global existence of classical solutions, the singularity formation of solutions, and the critical thresholds for either hyperbolic dissipative (damping) or relaxation systems. The smallness conditions on the initial or boundary data were given in most of these results. However, these results cannot be applied to our problem by the following facts. First, the $C^0$ and $C^1$ norms of the initial data in this paper can be large, which conflicts to their assumptions. Moreover, in some papers dealing with the compressible Euler equations in Lagrangian coordinates, they gave the assumptions on the smallness of $C^0$ or $C^1$ norms. It implies that the density is away from vacuum, which is different from our concern on the near-vacuum solutions. Second, our system given in \eqref{RIE} is not dissipative, and the source terms of the system do not have the relaxation effect. It makes this problem much difficult when we use the uniform a priori estimates to obtain the uniform $C^1$ norms of solutions. We note that, by letting the new unknown $\tilde{\rho} := a \rho$, equations in \eqref{a.8} can be written as the following system without the source term
\begin{equation}
\label{a.18}
\begin{cases}
\tilde{\rho}_t+(\tilde{\rho}v)_x=0,\\
v_t+(\frac{1}{2}v^2+\nu a^{1-\gamma}\tilde{\rho}^{\gamma-1})_x=0.
\end{cases}
\end{equation}
The shock formation of \eqref{a.18} was studied in \cite{Chen5}, but there are very limited results related to the global existence of classical solutions to \eqref{a.18}. The goal of this paper is to give a new result on the global existence of the classical solution near vacuum to problem \eqref{a.1} without small data of Riemann invariants, and study its asymptotic behavior.

In this paper, we write the first equation of \eqref{a.1} (or, equivalently, \eqref{a.8}) describing the conservation of mass as a hyperbolic equation with the source term. In order to obtain the global existence of the classical solution for problem \eqref{a.8}, we rewrite it as problem \eqref{RIE} in terms of Riemann invariants $S$ and $R$ defined by
\begin{equation}
\label{RI}
S(U) = v - 2\sqrt{\frac{\nu}{\gamma -1}} \rho^{\frac{\gamma-1}{2}} \ \ \text{and} \ \ \
R(U) = v + 2\sqrt{\frac{\nu}{\gamma -1}} \rho^{\frac{\gamma-1}{2}},
\end{equation}
where $U := [\rho, v]^T$ is the solution of initial-boundary value problem \eqref{a.8}. The work is based on the local existence, the maximum principle, and the uniform a priori estimates obtained by the generalized Lax transformations. With the aid of \cite{TLi3}, we establish the local existence of classical solutions of problem \eqref{RIE}. Furthermore, we find that $S$ is increasing along the first characteristic field and $R$ is decreasing along the second characteristic field. This leads to the maximum principle so that we have the uniform $C^0$ norms of $S$ and $R$. To obtain the uniform $C^1$ norms of $S$ and $R$, we provide a new version of uniform a priori estimates as follows. We use the generalized Lax transformations
\begin{equation*}
\label{a.Lax}
Y = e^{h} S_x + Q_1  \ \ \text{and} \ \ Z = e^{h} R_x + Q_2,
\end{equation*}
where $h = h(S,R)$, $Q_1 = Q_1(S,R)$, and $Q_2(S,R)$ are given in Subsection 2.2, and introduce a variable $\xi := (R/S) -1 > 0$ to transform the equations of $S_x$ and $R_x$ along the first and second characteristic curves, respectively, into two Riccati equations
\begin{equation*}
\label{a.19}
\dot{Y} (= Y_t + \lambda_{1} Y_x ) =  \mathcal{A} Y^2+\mathcal{B}Y+\mathcal{C} \ \ \text{and} \ \ \dot{Z} (= Z_t + \lambda_{2} Z_x ) =  \widehat{\mathcal{A}} Z^2+\widehat{\mathcal{B}}Z+\widehat{\mathcal{C}},
\end{equation*}
where $\mathcal{A} (= \widehat{\mathcal{A}}) <0$, $\widehat{\mathcal{B}}=\mathcal{B}+O(\xi)$, and $\widehat{\mathcal{A}}\widehat{\mathcal{C}}=\mathcal{A}\mathcal{C}+O(\xi)$. It means that the equation of $Z$ can be regarded as a perturbed equation of $Y$ when $\xi$ is sufficiently small. Consequently, when we have a suitable upper and lower bound for $Y$, another suitable upper and lower bound for $Z$ can also be obtained if $\xi$ is sufficiently small. For the equation of $Y$, we impose condition \eqref{k1} of $k(x)$, which is the restriction to the shape of the duct, such that $\mathcal{C} \geq 0$. In this case, the equation of $Y$ can be written as
\begin{equation*}
\label{a.20}
\dot{Y}=\mathcal{A} (Y-Y_1(t))(Y-Y_2(t))
\end{equation*}
for some $Y_1(t)\leq 0 \leq Y_2(t)$. There is an important information from Lemma \ref{key1} that, since $Y_2(t)$ may not be bounded, we give a more precise upper bound of $Y(t)$ depending on the integral of $\mathcal{A}(Y_2-Y_1)^2$, but not the supremum of $Y_2(t)$. Following from the generalized Lax transformation, we obtain in Theorem \ref{GC} that the bound of $S_x$ is controlled by the terms given in inequality \eqref{Sx}. Each of these terms can be carefully treated by the maximum principle and the behavior of $S$ and $R$ along the characteristic curves, which implies that the upper bound of $S_x$ is independent of time. Similar arguments also yield a time-independent bound of $R_x$. Therefore, the global existence of the classical solution to problem \eqref{RIE} is established if initial-boundary data of $S$, $R$, and $k(x)$ satisfy a set of conditions listed in Theorem \ref{GC}. Finally, writing $\rho_\eta$ and $v$ in \eqref{a.1} as the functions of $S$ and $R$, we can easily obtain the global existence of the classical solution to problem \eqref{a.1}.
To the asymptotic behavior of solutions, we can show by the proof of Theorem \ref{GC} that the inequalities for the initial-boundary data of $S$ and $R$ are preserved for all $(x,t) \in I \times [0, \infty)$. In particular, we have the key inequalities:
\begin{equation*}
\max \Big\{ S_t (x,t), R_t(x,t) \Big\} \leq 0 \leq \min \Big\{ S_x (x,t), R_x(x,t) \Big\}
\end{equation*}
for all $(x,t) \in I \times [0, \infty)$. Then, for any fixed $x \in I$, Riemann invariants $S(x,t)$, $R(x,t)$ and velocity $v(x,t)$ decrease to $0$, and density $\rho_{\eta} (x,t)$ tends to $0$ as $t \rightarrow \infty$. The convergence rates of $S(x,t)$, $R(x,t)$, $v(x,t)$ and $\rho^{(\gamma-1)/2}_{\eta} (x,t)$  are $O(1/t)$ as $t \rightarrow \infty$. In addition, if $k(x_{*}) > 0$ for some $x_{*} \in I$, then all the convergences are uniform on $[x_B, x_*]$.
We also prove that $S$ is increasing, and $R$, $\xi$ are decreasing along all characteristic curves. In particular, given a characteristic curve $\Gamma$, $\xi \rightarrow 0$ along $\Gamma$ if, and only if, $\rho_{\eta} \rightarrow 0$ along $\Gamma$.

The paper is organized as follows. In Section 2, we give the local existence of classical solutions to problem \eqref{RIE}, the maximum principle, and Riccati equations derived by the generalized Lax transformations. In Section 3, we establish the global existence of classical solutions to problem \eqref{RIE} and the asymptotic behavior of classical solutions. As an application, the results can be applied to the spherically symmetric solutions to $N$-dimensional compressible Euler equations. Finally, some numerical experiments are presented in Section 4 to support our theoretical results.

\section{Local existence, maximum principle, and Riccati equations for Riemann invariants.}
\setcounter{section}{2}

\subsection{Local existence and maximum principle.}

In this section, we start by studying the solution $U := [\rho, v]^T$ of the initial-boundary value problem \eqref{a.8} in terms of the Riemann invariants $R$ and $S$ defined in \eqref{RI}. In view of \eqref{RI}, it is clear to obtain that
\begin{equation}
\label{rhov}
\rho = \left( \frac{\gamma -1}{16 \nu} \right)^{\frac{1}{\gamma-1}} (R-S)^{\frac{2}{\gamma -1}} \ \ \text{and} \ \ v = \frac{S+R}{2}.
\end{equation}
Hence $S < R$ if, and only if, $\rho >0$. The eigenvalues of the Jacobian matrix of the flux of equations in \eqref{a.8} are
\begin{equation}
\label{eigenvalues}
\lambda_{1}(U) = v - \sqrt{\nu (\gamma -1)}\rho^{\frac{\gamma-1}{2}}   \ \ \text{and} \ \ \
\lambda_{2}(U) = v + \sqrt{\nu (\gamma -1)} \rho^{\frac{\gamma-1}{2}}.
\end{equation}
The gas flow is supersonic in $\Omega$ if $\lambda_{1}(U) >0$ or $\lambda_{2}(U) < 0$ for all $U \in \Omega$. We obtain from \eqref{RI} and \eqref{eigenvalues} that
\begin{equation}
\label{eigenRI}
\lambda_{1} =  \frac{\gamma + 1}{4}S +  \frac{3-\gamma}{4}R  \ \ \text{and} \ \ \
\lambda_{2} =  \frac{3-\gamma}{4}S + \frac{\gamma +1}{4}R .
\end{equation}
As a result, if $S >0$, then $R >0$ and hence $\lambda_{1} >0$ by \eqref{eigenRI} and $1 < \gamma < 3$. In other words, the gas flow is supersonic if $S >0$.

For simplicity, we let $\lambda_1(x,t)$, $\lambda_{2}(x,t)$ $S(x,t)$, and $R(x,t)$ denote $\lambda_1(U(x,t))$, $\lambda_2(U(x,t))$, $R(U(x,t))$, and $S(U(x,t))$, respectively. A direct calculation shows that problem \eqref{a.8} becomes
\begin{equation}
\label{RIE}
\begin{cases}
S_{t}(x,t) + \lambda_{1}(x,t) S_{x}(x,t) = g,\quad &x \in \mathring{I},\ t>0,\\
R_{t}(x,t) + \lambda_{2}(x,t) R_{x}(x,t) = -g,\quad &x \in \mathring{I},\ t>0,\\
S(x,0)=S_0(x),\quad R(x,0)=R_0(x),\quad &x \in I,\\
S(x_B,t)=S_B(t),\quad R(x_B,t)=R_B(t),\quad &t\ge 0,
\end{cases}
\end{equation}
where
\begin{equation}
\label{g}
g = g(x,S,R) = \frac{\gamma-1}{8} \frac{a'(x)}{a(x)} \left(R^2 - S^2 \right).
\end{equation}
We assume that the cross-sectional area of a duct $a(x)$ is a $C^2$ function satisfying $a>0$ and $a' \geq0$ on $I$, and then define
\begin{equation}
\label{k}
k(x) := \frac{a'(x)}{a(x)}
\end{equation}
so that $k(x) \geq 0$ is a $C^1$ function on $I$. This and $S \leq R$ imply that the function $g$ in \eqref{g} is nonnegative.  Consequently, equations in \eqref{RIE} tell us that $S(x,t)$ is increasing along the first characteristic curves
\begin{equation}
\label{Gamma1}
\Gamma^{1}:\frac{dx}{dt} = \lambda_{1} (S (x,t), R (x,t)),
\end{equation}
and $R(x,t)$ is decreasing along the second characteristic curves
\begin{equation}
\label{Gamma2}
\Gamma^{2}:\frac{dx}{dt} = \lambda_{2} (S (x,t), R (x,t)).
\end{equation}

The monotone property of the Riemann invariant $S$ (resp., $R$) along the first (resp., second) characteristic curves results in the maximum principle. Before giving the maximum principle, we first apply the work of Li and Yu in \cite{TLi3} to establish the local existence and uniqueness theorem to initial-boundary value problem \eqref{RIE}.

\begin{thm}[\cite{TLi3}]\label{TTLi}
Suppose that $k(x) \geq 0$, $S_0(x)$, and $R_0(x)$ are $C^1$ functions defined on $I$ with bounded $C^1$ norms and that $S_B(t)$ and $R_B(t)$ are $C^1$ functions defined on $[0, T]$ with bounded $C^1$ norms. Suppose also that $S_0(x_B)= S_B(0)$, $R_0(x_B) = R_B(0)$, and
\begin{equation*}
\begin{aligned}
& S'_{B}(0) + \Big( \frac{\gamma +1}{4} S_0 (x_B) + \frac{3-\gamma}{4} R_0 (x_B) \Big) S'_0 (x_B) = \frac{\gamma -1}{8} k(x_B) \Big( R_0^2(x_B) - S_0^2(x_B) \Big),
\\ & R'_{B}(0) + \Big( \frac{3-\gamma}{4} S_0 (x_B) + \frac{\gamma +1}{4} R_0 (x_B) \Big) R'_0(x_B) = - \frac{\gamma -1}{8} k(x_B) \Big( R_0^2(x_B) - S_0^2(x_B) \Big).
\end{aligned}
\end{equation*}
Then there exists a positive number $T' \leq T$ such that \eqref{RIE} admits a unique $C^1$ solution on $ I \times [0,T']$, where $T'$ depends only on the $C^1$ norms of $k$, $S_0$, $R_0$, $S_B$, and $R_B$.
\end{thm}

The following property addresses the maximum principle which is helpful to obtain the uniform $C^0$ norms of Riemann invariants $S$ and $R$.

\begin{lem}\label{2.1}
Suppose that $k(x) \geq 0$ on $I$, $0 < S_0(x) < R_0(x)$ on $I$, and $0 < S_B(t) < R_B(t)$ on $[0,T]$.
If \eqref{RIE} admits a unique $C^1$ solution on $I \times [0,T]$ and
$$M:=\max\left\{\sup_{x \in I} R_0(x), \sup_{t \in [0, T]} R_B(t)\right\}<\infty,$$
then
\begin{equation*}
\label{0SRc}
0 < S(x,t) < R(x,t) \leq M \quad\text{for all}\ (x,t)\in I \times [0,T].
\end{equation*}
\end{lem}

\begin{proof}
Since $S(x,t)$ is increasing along the first characteristic curves and $R(x,t)$ is decreasing along the second characteristic curves, it is sufficient to prove that $S(x,t) < R(x,t)$ on $I \times [0,T]$ or, equivalently, $\rho(x,t) > 0$ on $I \times [0,T]$. By \eqref{RI} and our assumptions, it is obvious that $\rho(x,0)> 0$ for all $x \in I$ and  $\rho(x_B,t)> 0$ for all $t \in [0, T]$.

If problem \eqref{RIE} admits a $C^1$ solution on $I \times [0,T]$, then $v_{x}$ is continuous on $I \times [0,T]$ based on the fact that $v = (S+R)/2$ given in \eqref{rhov}. We observe that the first equation in problem \eqref{a.8} can be written as
\begin{equation}
\label{arho}
(a\rho)_{t} + v (a\rho)_{x} + ( a\rho ) v_{x} = 0.
\end{equation}
Along the characteristic curves $\Gamma:dx/dt = v(x,t)$,
equation \eqref{arho} becomes
\begin{equation}
\label{muarho}
\frac{d}{dt} \Big( \mu(t) a(x(t)) \rho (x(t),t) \Big) = 0,
\end{equation}
where
\begin{equation*}
\mu(t) = \exp \Big( \int_{t_0}^{t} v_{x} (x(s),s) \ ds \Big)>0
\end{equation*}
if $\Gamma$ emanates from $(x_0,t_0)$. Solving equation \eqref{muarho} shows
\begin{equation}
\label{rhonu}
\rho(x(t),t) = \frac{\mu(t_0)}{\mu(t)} \frac{a(x_0)}{a(x(t))} \rho (x_0,t_0).
\end{equation}

For any given $(x^{*},t^{*}) \in I \times [0,T]$, there exists a unique characteristic curve $\Gamma$ that emanates from some point $(x_0^{*},t_0^{*})$ and passes through $(x^{*},t^{*})$. All the possibilities of $(x_0^{*},t_0^{*})$ are divided into two cases: $t_0^{*} = 0$ or $x_0^{*} = x_B$. If $t_0^{*} = 0$, then it follows from $\rho (x,0) > 0$ and $a(x) >0$ for all $x \in I$ that \eqref{rhonu} says $\rho(x^{*},t^{*}) >0$. If $x_0^{*} = x_B$, applying similar arguments to $x_0^{*} = x_B$ yields that $\rho(x^{*},t^{*})$ never vanishes. This completes the proof.
\end{proof}

Under the assumptions of Lemma \ref{2.1}, it follows easily from \eqref{eigenRI} that the wave speed $\lambda_{1} >0$ and hence the gas flow is supersonic.

With the assistance of Theorem \ref{TTLi}, in order to obtain the unique global classical solution for \eqref{RIE}, we need only to get the following uniform a priori estimates. That is, for any fixed $T>0$, if \eqref{RIE} admits a unique $C^1$ solution on $I \times [0,T]$, then the $C^0$ norms of $S$, $R$, $S_{x}$, and $R_{x}$ have upper bounds independent of $T$. Due to Lemma \ref{2.1}, it remains to show that the $C^0$ norms of $S_x$ and $R_x$ have upper bounds independent of $T$. To this end, we will derive two Riccati equations from \eqref{RIE} in the next subsection and then study the solutions of the Riccati equations in Section 3.

\subsection{Riccati equations.}

According to Lemma \ref{2.1}, we see that, if \eqref{RIE} admits a unique $C^1$ solution on $I \times [0,T]$, then $R-S > 0$ on $I \times [0,T]$. Thus, $1/(R-S)$ and $\ln (R-S)$ are well-defined, which will be used in the following discussion. We now derive the Riccati equations from \eqref{RIE} by the generalized Lax transformations:
\begin{equation}
\label{YZ}
Y = e^h S_{x} + Q_1   \ \ \text{and} \ \ \
Z = e^h R_{x} + Q_2,
\end{equation}
where $h = h(S,R)$, $Q_1 = Q_1 (S,R)$, and $Q_2 = Q_2(S,R)$. Then $Y$ and $Z$ satisfy
\begin{equation*}
\label{Y}
\begin{aligned}
Y_t+\lambda_1 Y_x= \mathcal{A}Y^2 + \mathcal{B}Y + \mathcal{C} + \mathcal{D}Z + \mathcal{E}YZ  \ \ \text{and} \ \
Z_t+\lambda_2 Z_x = \widehat{\mathcal{A}}Z^2 +\widehat{ \mathcal{B}}Z + \widehat{\mathcal{C}} + \widehat{\mathcal{D}}Y + \widehat{\mathcal{E}}YZ,
\end{aligned}
\end{equation*}
respectively, where
\begin{equation*}
\label{A01}
\begin{aligned}
& \mathcal{D} = \frac{\gamma -1}{4} k R + \frac{3 - \gamma}{4} e^{-h} Q_1 + \frac{\gamma -1}{2} e^{-h} (R-S) Q_1\frac{\partial h}{\partial R} - \frac{\gamma -1}{2} e^{-h} (R-S) \frac{\partial Q_1}{\partial R},
\\ & \mathcal{E} = - \frac{3 - \gamma}{4} e^{-h}  - \frac{\gamma - 1}{2} e^{-h} (R-S) \frac{\partial h}{\partial R},
\\ & \widehat{\mathcal{D}} = \frac{\gamma -1}{4} k S + \frac{3 - \gamma}{4} e^{-h} Q_2 - \frac{\gamma -1}{2} e^{-h} (R-S) Q_2\frac{\partial h}{\partial S} + \frac{\gamma -1}{2} e^{-h} (R-S) \frac{\partial Q_2}{\partial S},
\\ & \widehat{\mathcal{E}} = - \frac{3 - \gamma}{4} e^{-h}  + \frac{\gamma - 1}{2} e^{-h} (R-S) \frac{\partial h}{\partial S},
\end{aligned}
\end{equation*}
and the other coefficients will be presented after solving $h$, $Q_1$, and $Q_2$ by letting
$\mathcal{D}= \mathcal{E}=\widehat{\mathcal{D}}=\widehat{\mathcal{E}}=0$.

Let $\mathcal{E} = \widehat{\mathcal{E}} = 0$. Then we get that
\begin{equation}
\label{h}
h = b \ln  \left(R-S \right),\qquad b := - \frac{3 - \gamma}{2( \gamma -1)}.
\end{equation}
Since $1 < \gamma < 3$, we have $b<0$. Using \eqref{h} and setting $\mathcal{D} = \widehat{\mathcal{D}} = 0$ yield that, if $b \neq -1$, then
\begin{equation}
\label{bn-1}
\begin{aligned}
 & Q_1 = \frac{k}{2b} S (R-S)^{b} +  \frac{k}{2(b+1)} (R-S)^{b+1} + G_1(S),
\\ & Q_{2} = \frac{k}{2b} R (R-S)^{b} - \frac{k}{2(b+1)} (R-S)^{b+1} + G_2(R),
\end{aligned}
\end{equation}
for some functions $G_1(S)$ and $G_2(R)$; if $b = -1$, then
\begin{equation*}
\begin{aligned}
 & Q_1 = -\frac{k}{2} S (R-S)^{-1} +  \frac{k}{2} \ln (R-S) + H_1(S),
\\ & Q_{2} = -\frac{k}{2} R (R-S)^{-1} - \frac{k}{2} \ln (R-S) + H_2 (R),
\end{aligned}
\end{equation*}
for some functions $H_1(S)$ and $H_2(R)$. In the near-vacuum case, $G_1$, $G_2$, $H_1$, and $H_2$ are small compared to the other terms of $Q_1$ and $Q_2$ and hence they can be ignored.

We let $b \neq -1$ and define
\begin{equation}
\label{C0b}
\aligned
\mathcal{C}_1(S,R,k,k',b) &:= \frac{ k^2}{8 b^2(b+1)^2(1-2b)}\Big(b(1-b)^2R^2+2b(b^2+3b-2)RS+(b^3+2b^2+3b-2)S^2\Big)\\
&\quad +\frac{k'}{4 b(b+1)(1-2b)}\Big( b(1-b)R^2-2b^2RS+(2-3b-b^2)S^2\Big),\\
\widehat{\mathcal{C}}_1(S,R,k,k',b) &:= \frac{ k^2}{8 b^2(b+1)^2(1-2b)}\Big(b(1-b)^2S^2+2b(b^2+3b-2)RS+(b^3+2b^2+3b-2)R^2\Big)\\
&\quad +  \frac{k'}{4 b(b+1)(1-2b)}\Big( b(1-b)S^2-2b^2RS+(2-3b-b^2)R^2 \Big),\\
\mathcal{C}_1(S,R,k,k',-1) &:= \frac{ k^2}{24} \Big(-2R (R-S) \ln(R-S) + 8S (R-S) \ln(R-S) -4 (R-S)^2 (\ln (R-S) )^2\\
&\quad  -3R^2 -3S^2 \Big) +\frac{k'}{24} \Big( 2(R^2 - S^2) - (4R+8S)(S - (R-S) \ln(R-S) ) \Big),\\
\widehat{\mathcal{C}}_1(S,R,k,k',-1) &:= \frac{ k^2}{24} \Big( 2S (R-S) \ln(R-S) - 8R (R-S) \ln(R-S) - 4 (R-S)^2 (\ln (R-S) )^2\\
&\quad -3S^2 -3R^2 \Big) +\frac{k'}{24} \Big( 2(S^2 - R^2) - (4S+8R)(R + (R-S) \ln(R-S) ) \Big).
\endaligned
\end{equation}
Then along the first characteristic curves $\Gamma^1$, $Y$ satisfies the following Riccati equation:
\begin{equation}
\label{Y07}
\frac{d}{dt}Y = \mathcal{A}Y^2 + \mathcal{B}Y + \mathcal{C},
\end{equation}
where when $b\ne -1$,
\begin{equation}
\label{A02}
\begin{aligned}
& \mathcal{A} = - \frac{1-b}{1-2b} (R-S)^{-b} ,
\\ & \mathcal{B} = - \frac{k}{2b(b+1)(1-2b)} \Big( b (b^2 +3b -2) R + (b^3 + 2b^2 +3b -2 ) S\Big),
\\ & \mathcal{C} = (R-S)^{b} \mathcal{C}_1(S,R,k,k',b),
\end{aligned}
\end{equation}
and when $b=-1$,
\begin{equation*}
\begin{aligned}
 \mathcal{A} &= - \frac{2}{3} (R-S),
\\ \mathcal{B} &= \frac{k}{6} \Big( R  - 4S + 4(R-S) \ln (R-S) \Big),
\\ \mathcal{C} &= (R-S)^{-1} \mathcal{C}_1(S,R,k,k',-1).
\end{aligned}
\end{equation*}
Similarly, along the second characteristic curves $\Gamma^2$, $Z$ satisfies the following Riccati equation:
\begin{equation}
\label{Z01}
\frac{d}{dt}Z =\widehat{\mathcal{A}}Z^2 + \widehat{\mathcal{B}}Z + \widehat{\mathcal{C}},
\end{equation}
where when $b\ne -1$,
\begin{equation}
\label{A03}
\begin{aligned}
& \widehat{\mathcal{A}} = - \frac{1-b}{1-2b}(R-S)^{-b},
\\ & \widehat{\mathcal{B}} = - \frac{k}{2b(b+1)(1-2b)} \Big( b (b^2 +3b -2) S + (b^3 + 2b^2 +3b -2 ) R \Big),
\\ & \widehat{\mathcal{C}} =(R-S)^{b} \widehat{\mathcal{C}}_1(S,R,k,k',b),
\end{aligned}
\end{equation}
and when $b=-1$,
\begin{equation*}
\begin{aligned}
\widehat{\mathcal{A}} &= - \frac{2}{3} (R-S),
\\ \widehat{\mathcal{B}} &= \frac{k}{6} \Big( S  - 4R - 4(R-S) \ln (R-S) \Big),
\\ \widehat{\mathcal{C}} &= (R-S)^{-1} \widehat{\mathcal{C}}_1(S,R,k,k',-1).
\end{aligned}
\end{equation*}

We observe that $\mathcal{A} <0$. For $\mathcal{C} \geq 0$, the quadratic form on the right-hand side of \eqref{Y07} can be written into the form $\mathcal{A}(Y-Y_1(t))(Y-Y_2(t))$, where $Y_1(t)\le 0\le Y_2(t)$. A similar argument also works for $Z$. This motivates us to consider the following definition.

\begin{defn}\label{def1}
Let $W=W_{1}(t)$ and $W=W_{2}(t)$, $t\in [T_1,T_2]$, be two parametrized curves in the $t$-$W$ plane. If there exists a horizontal line $W = W_{*}$ such that $W_{1}(t) \leq W_{*} \leq W_{2} (t)$ for all $t \in [T_1, T_2]$, we say that $W = W_{*}$ is a horizontal separating line between $W=W_1(t)$ and $W=W_2(t)$.
\end{defn}

The next lemma locates the solution of the Riccati equation if $W = 0$ is a horizontal separating line between $W=W_1(t)$ and $W=W_2(t)$.

\begin{lem}\label{key1}
Consider the Riccati equation
\begin{equation}
\label{ricW}
\frac{dW(t)}{dt} = \mathcal{A}(t)(W(t) - W_1(t))(W(t) - W_2(t)), \ \ t \in [T_1, T_2],
\end{equation}
where $\mathcal{A}(t) < 0$ and $W_1(t) \leq  0 \leq W_2(t)$ for all $ t \in [T_1, T_2]$.  If $W(T_1) \geq 0$, then
\begin{equation*}
\label{W}
0 \leq W(t) \leq W(T_1) - \frac{1}{4} \int_{T_1}^{t} \mathcal{A}(s) \left(W_2(s) - W_1(s)\right)^2 \ ds, \ \ t \in [T_1,T_2].
\end{equation*}
\end{lem}

\begin{proof}
Let $ X(t)$ be a differentiable function satisfying
\begin{equation}
\label{ricX}
\frac{dX(t)}{dt} \leq \mathcal{A}(t)(X(t) - W_1(t))(X(t) - W_2(t)), \ \ t \in [T_1, T_2].
\end{equation}
Since $W(T_1) \geq 0$ and $X(t)\equiv 0$ is a solution of \eqref{ricX}, applying the theory of differential inequalities yields that
\begin{equation*}
0=X(t) \leq W(t), \ \ t \in  [T_1,T_2].
\end{equation*}

On the other hand, we observe that the quadratic form on the right hand side of \eqref{ricW} attains its maximum when $W\equiv (W_1+W_2)/2$.
As a result, we obtain that
\begin{equation*}
	\frac{dW(t)}{dt} \leq  - \frac{1}{4} \mathcal{A}(t)  (W_2(t) - W_1(t))^2, \ \ t \in [T_1,T_2].
\end{equation*}
Integrating both sides of the above inequality finishes the proof.
\end{proof}

\section{Existence and uniqueness of the global classical solution.}
\setcounter{section}{3}

In this section, we investigate the global in time existence and uniqueness of the classical solution for \eqref{RIE}.
We employ Lemma \ref{key1} to get the uniform a priori estimates for $R_x$ and $S_x$ under some suitable initial and boundary conditions.
Throughout this section, we suppose that $k(x)$, $S_0$, $R_0$, $S_B$, and $R_B$ satisfy the following conditions:
\begin{itemize}
    \item[(A1)] $k(x) (\geq 0)  \in C^1(I)$ has a bounded $C^1$ norm.
    \item[(A2)] $S_0(x_B)= S_B(0)$, $R_0(x_B) = R_B(0)$, and
    \begin{equation*}
    \begin{aligned}
& S'_{B}(0) + \Big( \frac{\gamma +1}{4} S_0 (x_B) + \frac{3-\gamma}{4} R_0 (x_B) \Big) S'_0 (x_B) = \frac{\gamma -1}{8} k(x_B) \Big( R_0^2(x_B) - S_0^2(x_B) \Big),
\\ & R'_{B}(0) + \Big( \frac{3-\gamma}{4} S_0 (x_B) + \frac{\gamma +1}{4} R_0 (x_B) \Big) R'_0(x_B) = - \frac{\gamma -1}{8} k(x_B) \Big( R_0^2(x_B) - S_0^2(x_B) \Big).
\end{aligned}
    \end{equation*}
	\item[(A3)] $S_0(x), R_0(x) \in C^1(I)$ have bounded $C^1$ norms, and $0 < S_0(x) < R_0(x)$ for all $x \in I$.
    \item[(A4)] $S_B(t), R_B(t) \in C^1([0, \infty))$ have bounded $C^1$ norms, and $0 < S_B(t) < R_B(t)$ for all $t \in [0, \infty)$.
\end{itemize}

\subsection{A divergent duct with a bounded cross-sectional area.}

Although the coefficients in \eqref{Y07} and \eqref{Z01} are extremely complicated, there is a conceivable way out of this quandary. We observe that they are formally symmetric to some extent. For $b \neq -1$, we write $\mathcal{B} = \mathcal{B}(S,R)$ in \eqref{A02} and $\widehat{\mathcal{B}} = \widehat{\mathcal{B}}(S,R)$ in \eqref{A03}. Then it is clear that $ \mathcal{B}(S,R) = \widehat{\mathcal{B}}(R,S)$ and $\mathcal{C}_1(S,R,k,k',b) =\widehat{ \mathcal{C}}_1(R,S,k,k',b)$ defined in \eqref{C0b}. Analogous arguments can be dealt with for $b = -1$. Moreover, since the coefficients in \eqref{Y07} and \eqref{Z01} are much simpler when $R \approx S$, i.e., $\rho_{\eta} \approx 0$, it inspires us to study the near-vacuum solutions.
Let
\begin{equation}
\label{xi0}
\xi(x,t) :=\frac{R(x,t)}{S(x,t)} - 1 = 4 \sqrt{ \frac{\nu}{\gamma -1} } \ \frac{ \rho^{\frac{\gamma-1}{2}} (x,t)}{S(x,t)}.
\end{equation}
Under conditions ${\rm (A1)}$,  ${\rm (A3)}$ and ${\rm (A4)}$, Lemma \ref{2.1} points out that if \eqref{RIE} admits a unique $C^1$ solution on $I \times [0,\infty)$, then $\xi(x,t) >0$ on $I \times [0,\infty)$. Actually, $\xi(x,t)$ is dominated by $\sqrt{\nu}$ if we further suppose that
\begin{itemize}
	\item[(A5)] $S_0'(x) \geq 0$, $R'_0(x) \geq 0$, $\xi_0(x) = O(\sqrt{\nu})$ for all $x \in I$, and $S_B'(t) \leq 0$, $R_B'(t) \leq 0$, $\xi_B(t) = O(\sqrt{\nu})$ for all $t \in [0,\infty)$, where
\begin{equation*}
\label{xi0B}
\xi_0(x):=\xi(x,0) \ \text{for all} \ x \in I \ \ \text{and} \ \ \xi_B(t):=\xi(x_B,t)\ \text{for all} \ t \in [0,\infty).
\end{equation*}
\end{itemize}

\begin{lem}\label{estimate}
Suppose that conditions ${\rm (A1)}$--${\rm (A5)}$ hold. If \eqref{RIE} admits a unique $C^1$ solution on $I \times [0,\infty)$, then
\begin{equation}
\label{nu2}
0 < \xi(x,t) = O(\sqrt{\nu})
\end{equation}
for all $(x,t) \in I \times [0,\infty)$.
\end{lem}

\begin{proof}
Let $(x,t)$ be any given point in $I \times [0,\infty)$. Then the characteristic curves $\Gamma^i$ ($i=1,2$) passing through $(x,t)$ emanate from some point of the form either $(x_i,0)$ or $(x_B,t_i)$.

If $\Gamma^1$ and $\Gamma^2$ emanate from $(x_1,0)$ and $(x_2,0)$, respectively, then it is clear that $x_2 < x_1$. Since $S$ (resp., $R$) is increasing (resp., decreasing) along $\Gamma^1$ (resp., $\Gamma^2$), by the definition of $\xi$ we have
\begin{equation*}
0< \xi(x,t) = \frac{R(x,t)}{S(x,t)} - 1 \leq  \frac{R(x_2, 0)}{S(x_1, 0)} - 1.
\end{equation*}
Since $S'_0(x) \geq 0$ and $R'_0(x) \geq 0$ for all $x \in I$ in ${\rm (A5)}$, it follows from $x_2 < x_1$ that
\begin{equation}
\label{xi1}
0< \xi(x,t) \leq  \min {( \xi_0(x_1), \xi_0(x_2) )}.
\end{equation}
Condition ${\rm (A5)}$ also gives
\begin{equation}
\label{xi2}
\min {(\xi_0(x_1), \xi_0(x_2))} =  O(\sqrt{\nu}).
\end{equation}
An easy consequence of \eqref{xi1} and \eqref{xi2} shows estimate \eqref{nu2}. If $\Gamma^1$ emanates from $(x_1,0)$ and $\Gamma^2$ from $(x_B, t_2)$, or $\Gamma^1$ emanates from $(x_B,t_1)$ and $\Gamma^2$ from $(x_B, t_2)$, then similar arguments lead to estimate \eqref{nu2}.
\end{proof}

By introducing the parameter $\xi$, the expressions for $\mathcal{C}$ and $\widehat{\mathcal{C}}$ have the same simplified form. In fact, we have
\begin{equation}
\label{A04}
\mathcal{C}, \ \widehat{\mathcal{C}} = S^{b+2} \xi^{b} \left( \Big(\frac{-1}{4b^2} +O(\xi) \Big) k^2 +  \Big(\frac{1}{2b} + O(\xi) \Big) k' \right)
\end{equation}
for all $b <0$. From \eqref{A04}, we see that $\mathcal{C}$ can be viewed as a function of $\xi$ and thus a function of $\nu$ via Lemma \ref{estimate}.
In the following, we seek an additional requirement for $k$ which guarantees that $\mathcal{C}$ and $\widehat{\mathcal{C}}$ are nonnegative.

\begin{lem}\label{key2}
Suppose that conditions ${\rm (A1)}$--${\rm (A5)}$ hold. If there is a number $\delta \in (0,2)$ such that $k(x)$ satisfies
\begin{equation}\
\label{k1}
k'(x) \leq \frac{k^2(x)}{(2 - \delta)b}
\end{equation}
for all $ x \in I$, then there is a parameter $ \nu_{*}>0$ such that $\mathcal{C} = \mathcal{C}(\nu) \geq 0$ and $\widehat{\mathcal{C}} = \widehat{\mathcal{C}}(\nu) \geq 0$ for $\nu \leq \nu_{*}$.
\end{lem}

\begin{proof}
By \eqref{A04}, $\mathcal{C}  = \mathcal{C}(\nu) \geq 0$ is equivalent to
\begin{equation}
\label{A05}
\Big(\frac{-1}{4b^2} + O(\xi) \Big) k^2 +  \Big(\frac{1}{2b} + O(\xi) \Big) k'   \geq 0.
\end{equation}
It follows from \eqref{k1} that there is a parameter $\xi_{*}>0$, i.e., $\nu_{*}>0$ by Lemma \ref{estimate}, such that \eqref{A05} holds.
The same arguments as above show $\widehat{\mathcal{C}} = \widehat{\mathcal{C}}(\nu) \geq 0$ for $\nu \leq \nu_{*}$. Hence the proof is complete.
\end{proof}

Suppose that \eqref{RIE} admits a unique $C^1$ solution on $I \times [0,T]$.
Lemma \ref{key2} states that if $k$ satisfies \eqref{k1}, then as $\nu$ is sufficiently small, the quadratic form on the right-hand side of \eqref{Y07} can be split into the form $\mathcal{A}(Y-Y_1(t))(Y-Y_2(t))$ with $Y_1(t)\le Y_2(t)$ for $t\in [0,T]$ and $Y=0$ forms a horizontal separating line between $Y=Y_1(t)$ and $Y=Y_2(t)$.
If $\Gamma^1$ emanates from the point $(x(0),0)$ and $Y(0) \geq 0$, then Lemma \ref{key1} shows that $Y(t) \geq 0$ for all $t\in [0,T]$.
On the other hand, if $\Gamma^1$ emanates from the point $(x_B,t_B)$ and $Y(t_B) \geq 0$, Lemma \ref{key1} also gives that $Y(t) \geq 0$ for all $t\in [t_B,T]$. As for $Z$, we can apply the same arguments to $Z(0) \geq 0$ or $Z(t_B) \geq 0$. As a consequence, the conditions $Y(0)\ge 0$ and $Z(0)\ge 0$ provide suitable initial conditions for \eqref{RIE} while the conditions $Y(t_B)\ge 0$ and $Z(t_B) \geq 0$ provide suitable boundary conditions for \eqref{RIE}. Define
\begin{equation*}
\epsilon_0(x):=\ln \big( S_0(x)\xi_0(x) \big) \  \text{for all} \ x \in I \ \ \text{and} \ \
\epsilon_B(t):=\ln \big( S_B(t)\xi_B(t) \big) \ \text{for all} \ t \in [0, \infty).
\end{equation*}
Then the initial and boundary conditions are given in the following two lemmas.

\begin{lem}\label{key3}
Suppose that conditions ${\rm (A1)}$--${\rm (A5)}$ hold. Then $Y(0) \geq 0$ if, and only if,
\begin{equation}
\label{S0}
\begin{cases}
\frac{k(x)S_0(x)}{2} \left( \frac{1}{-b} - \frac{1}{b+1}\xi_0(x) \right) \leq S'_{0}(x), & \text{if} \ b \neq -1, \\
\frac{k(x)S_0(x)}{2} \Big(  1 - \epsilon_0(x) \xi_0(x) \Big) \leq S'_{0}(x), & \text{if} \ b = -1,
\end{cases}
\end{equation}
for all $x \in I$. And $Z(0) \geq 0$ if, and only if,
\begin{equation}
\label{R0}
\begin{cases}
\frac{k(x)S_0(x)}{2} \left(  \frac{1}{-b} - \frac{1}{b(b+1)}\xi_0(x) \right) \leq R'_{0}(x), & \text{if} \ b \neq -1, \\
\frac{k(x)S_0(x)}{2} \Big(  1 + \xi_0(x) +\epsilon_0(x) \xi_0(x) \big)  \Big) \leq R'_{0}(x), & \text{if} \ b = -1,
\end{cases}
\end{equation}
for all $x \in I$. In particular, if
\begin{equation}
\label{k2}
\frac{k(x)S_0(x)}{-b} \leq S'_{0}(x)\quad \text{and}\quad\frac{k(x)R_0(x)}{-b} \leq R'_{0}(x)
\end{equation}
for all $x \in I$, then there is a parameter $\nu_{*}>0$ such that $Y(0) \geq 0$ and $Z(0) \geq 0$ for $\nu \leq \nu_{*}$.
\end{lem}

\begin{proof}
Let $\Gamma^1$ be the first characteristic curve that emanates form $(x,0)$. Employing the definition of $Y$ in \eqref{YZ}, we get, for $b \neq -1$, that $Y(0) \geq 0$ is the same as
\begin{equation}
\label{Y(0)}
(R_0(x)-S_0(x))^b \left(S'_{0}(x) + \frac {k(x)S_0(x)}{2b} + \frac{k(x)(R_0(x)-S_0(x))}{2(b+1)} \right) \geq 0.
\end{equation}
Since $(R_0(x) - S_0 (x))^b>0$ by condition ${\rm (A3)}$, it is obvious that $Y(0) \geq 0$ if, and only if,
\begin{equation*}
\label{Y(0)1}
S'_{0}(x) + \frac {k(x)S_0(x)}{2b} + \frac{k(x)(R_0(x)-S_0(x))}{2(b+1)}\geq 0.
\end{equation*}
This is equivalent to the first inequality in \eqref{S0} by applying the definition of $\xi$.
In light of Lemma \ref{estimate}, we may choose $\nu_{*}>0$ small enough such that as $\nu\le \nu_{*}$,
the first inequality in \eqref{k2} implies that $Y(0)\ge 0$.
One can apply similar processes for the case $b=-1$ and for $Z$ to derive the remaining parts of \eqref{S0}--\eqref{k2}.
We omit the details.
\end{proof}

If we consider that $\Gamma^1$ or $\Gamma^2$ emanates from  $(x_B, t_B)$, then the next lemma immediately follows from Lemma \ref{key3} and equations $S_t = g - \lambda_1 S_x$ and $R_t = -g - \lambda_2 R_x$ in \eqref{RIE}. We also omit the similar proof.

\begin{lem}\label{key4}
Suppose that conditions ${\rm (A1)}$--${\rm (A5)}$ hold. Then $Y(t_B) \geq 0$ if, and only if,
\begin{equation}
\label{SB}
\begin{cases}
S'_{B}(t) \leq \frac{k(x_B)S^2_B(t)}{2} \left(  \frac{1}{b} +  \frac{1}{b+1}\xi_B(t) + \frac{1-b}{2(b+1)(1-2b)}\xi^2_B(t) \right), & \text{if} \ b \neq -1, \\
S'_{B}(t) \leq \frac{k(x_B)S^2_B(t)}{2} \Big( - 1  +  \epsilon_B(t)\xi_B(t) + \frac{2\epsilon_B(t)+1}{6}\xi^2_B(t) \Big), & \text{if} \ b = -1,
\end{cases}
\end{equation}
for all $t \in [0,\infty)$. And $Z(t_B) \geq 0$ if, and only if,
\begin{equation}
\label{RB}
\begin{cases}
R'_{B}(t) \leq \frac{k(x_B)S^2_B(t)}{2} \left(  \frac{1}{b} +  \frac{b+2}{b(b+1)} \xi_B(t) - \frac{b^2+3b-2}{2b(b+1)(1-2b)} \xi^2_B(t) \right), & \text{if} \ b \neq -1, \\
R'_{B}(t) \leq  \frac{k(x_B)S^2_B(t)}{2} \Big(  -1 - (\epsilon_B(t)+2) \xi_B(t)  - \frac{4 \epsilon_B(t) + 5}{6}  \xi^2_B(t) \Big), & \text{if} \ b = -1,
\end{cases}
\end{equation}
for all $t \in [0,\infty)$. In particular, if
\begin{equation}
\label{k3}
S'_{B}(t) \leq \frac{k(x_B)S^2_B(t)}{b}\quad\text{and}\quad R'_{B}(t) \leq \frac{k(x_B)R^2_B(t)}{b}
\end{equation}
for all $t \in [0,\infty)$, then there is a parameter $\nu_{*}>0$ such that $Y(t_B) \geq 0$ and $Z(t_B) \geq 0$ for $\nu \leq \nu_{*}$.
\end{lem}

\begin{rem}\label{remark1}
Examining inequality \eqref{k1}, there is no $C^1$ function $k$ defined on $\mathbb{R}$ such that inequality \eqref{k1} holds unless $k(x) \equiv 0$ for all $x \in \mathbb{R}$. Nevertheless, there indeed exists a $C^1$ function $k$ defined on an interval of the form $[x_B, \infty)$ such that inequality \eqref{k1} holds. For example, $k(x) = 1/(2x^2 -x)$ and $k(x) = 1/x^2$ or, equivalently, $a(x) = 2 - (1/x)$ and $a(x) =e^{1-(1/x)}$ by recalling $k(x) = a'(x)/a(x)$. This is the reason why we need to consider the initial-boundary value problem \eqref{RIE}. Moreover, using condition ${\rm (A3)}$, inequalities in \eqref{k1} and \eqref{S0}--\eqref{R0} imply that as $\nu$ is sufficiently small, we have
\begin{equation*}
\lim_{x\rightarrow \infty} k(x) = 0 \quad \text{and} \quad \int^\infty_{x_B} k(x)dx<\infty,
\end{equation*}
respectively. In other words, the cross-sectional area of a duct $a(x)$ meets the following restrictions:
\begin{equation*}
	\lim_{x\rightarrow \infty} a(x)  \ \text{exists} \quad \text{and} \ \lim_{x\rightarrow \infty} a'(x) = 0.
\end{equation*}
\end{rem}

With the aid of Lemmas \ref{key1} and \ref{estimate}--\ref{key4}, we are ready to derive the uniform a priori estimates for $R_{x}$ and $S_{x}$, which together with Theorem \ref{TTLi} and Lemma \ref{2.1} gives the following theorem.

\begin{thm}\label{GC}
Let $k(x)$ satisfy condition ${\rm (A1)}$ and inequality \eqref{k1}. Suppose that $S_0(x)$, $R_0(x)$, $S_B(t)$, $R_B(t)$ satisfy conditions ${\rm (A2)}$--${\rm (A5)}$, \eqref{S0}, \eqref{R0}, \eqref{SB}, and \eqref{RB}. Then \eqref{RIE} admits a unique global $C^1$ solution for sufficiently small $\nu$.
\end{thm}

\begin{proof}
By Theorem \ref{TTLi} and Lemma \ref{2.1}, it suffices to prove that, for any fixed $T>0$, if \eqref{RIE} admits a unique $C^1$ solution on $I \times [0,T]$, then the $C^0$ norms of $S_{x}$ and $R_{x}$ have upper bounds independent of $T$. We only show the case that $S_x$ is bounded on $I \times [0,T]$ with a bound independent of $T$ for $b\ne -1$ since the other cases can be done in a similar way.

Let $(x,t)$ be any given point in $I \times (0,T]$.
Then the first characteristic curve $\Gamma^1$ passing through $(x,t)=(x(t),t)$ emanates from either $(x_0,0)=(x(0),0)$ or $(x_B,t_B)=(x(t_B),t_B)$.
We only prove for the first case since the second one can be treated similarly. From Lemmas \ref{key2}--\ref{key3}, we have $\mathcal{C}(\nu) \geq 0$
and $Y(0) \geq 0$ for sufficiently small $\nu$. Lemma \ref{key1} tells us that
\begin{equation}
\label{Yi01}
0 \leq  Y(t) \leq Y(0) -  \frac{1}{4} \int_{0}^{t}  \mathcal{A}(x(s),s) \left(Y_2(s) - Y_1(s) \right)^2 \ ds.
\end{equation}
Since $Y_2 - Y_1 = \sqrt{\mathcal{B}^2 - 4\mathcal{A}\mathcal{C}}/(-\mathcal{A})$, the inequalities in \eqref{Yi01} can be rewritten as
\begin{equation}
\label{Yi02}
0 \leq  Y(t) \leq Y(0) -  \frac{1}{4} \int_{0}^{t} \frac{\mathcal{B}^2(x(s),s)}{\mathcal{A}(x(s),s)}  - 4 \mathcal{C} (x(s),s) \ ds.
\end{equation}
Substituting the expressions for $\mathcal{A}, \mathcal{B}$, and $\mathcal{C}$ in \eqref{A02} into the integral in \eqref{Yi02} and using the definition of $\xi$ in \eqref{xi0}, a simple manipulation shows that
\begin{equation}
\label{Yi03}
0 \leq  Y(t) \leq Y(0) + J_1(t) + J_2(t),
\end{equation}
where
\begin{equation*}
J_1(t):= \int_{0}^{t} \Big( (R-S)(x(s),s) \Big)^{b} k^2(x(s)) S^2(x(s),s) \bigg(-\frac{2b+7}{16(1-b)} + O(\xi)\bigg) \ ds
\end{equation*}
and
\begin{equation*}
J_2(t):= \int_{0}^{t} \Big( (R-S)(x(s),s) \Big)^{b} k'(x(s)) S^2(x(s),s) \bigg(\frac{1}{2b} + O(\xi)\bigg) \ ds.
\end{equation*}
From the definition of $Y$ in \eqref{YZ}, we see that \eqref{Yi03} is equivalent to
\begin{equation}
\label{Sx}
-\left((R-S)^{-b} Q_1\right)(x,t) \leq  S_{x}(x,t) \leq (R-S)^{-b}(x,t) \big( Y(0) + J_1(t) + J_2(t) - Q_1(x,t) \big).
\end{equation}

We now prove that each of $(R-S)^{-b} Q_1$, $(R-S)^{-b} Y(0)$, $(R-S)^{-b} J_1$, and $(R-S)^{-b} J_2$ has a bound on $I \times [0,T]$ independent of $T$ as follows.
Using the expression for $Q_1$ in \eqref{bn-1} yields that
\begin{equation}
\label{Q1}
\left((R-S)^{-b} Q_1\right) (x,t) = \frac{k(x)S(x,t)}{2b} + \frac{k(x)(R-S)(x,t)}{2(b+1)}.
\end{equation}
By conditions ${\rm (A1)}$, ${\rm (A3)}$, and ${\rm (A4)}$, Lemma \ref{2.1} says that there exists a constant $M>0$ independent of $T$ such that
\begin{equation}
\label{M}
0 < S(x,t) < R(x,t) \leq M  \quad \text{for all} \ (x,t) \in I\times [0,T],
\end{equation}
which together with condition ${\rm (A1)}$ again show that $(R-S)^{-b} Q_1$ has a bound on $I \times [0,T]$ independent of $T$.

To estimate $(R-S)^{-b} Y(0)$, we observe that
\begin{equation*}
(R-S)^{-b}(x,t) Y(0) = \left( \frac{(R-S)(x,t)}{(R_0 - S_0)(x_0)} \right)^{-b}  \left(S'_{0} (x_0) + \frac {(kS_0)(x_0)}{2b} + \frac{(k(R_0-S_0))(x_0)}{2(b+1)} \right).
\end{equation*}
Using the definition of $\xi$, the left inequality in \eqref{Sx} and equality \eqref{Q1} exhibit
\begin{equation}
\label{Sx>0}
- ((R-S)^{-b} Q_1)(x,t) = \frac{k(x)S(x,t)}{2} \left( \frac{1}{-b} - \frac{\xi(x,t)}{b+1}\right) \leq S_x.
\end{equation}
From Lemma \ref{estimate} and \eqref{Sx>0}, we find that if $\nu$ is sufficiently small, then $S_x \geq 0$ on $[0,T]$. Similarly, $R_x \geq 0$ on $[0,T]$ as we take $Z$ into consideration. In view of \eqref{RIE},
we have
\begin{equation*}
S_t + \lambda_{i} S_x = (\lambda_i - \lambda_1) S_x + g \quad \text{and} \quad R_t +  \lambda_{i} R_x = ( \lambda_{i} -  \lambda_{2}) R_x - g \quad \text{for}\ i=1,2,
\end{equation*}
which implies that $S$ is increasing and $R$ is decreasing along all characteristic curves $\Gamma^{1}$ and $\Gamma^{2}$ given in \eqref{Gamma1} and \eqref{Gamma2}, respectively. In particular,
\begin{equation}
\label{leq1}
(R-S)(x(s_2),s_2) \leq (R-S)(x(s_1),s_1)\quad \text{for all}\ 0\le s_1\le s_2\le T.
\end{equation}
Hence we infer from conditions ${\rm (A1)}$, ${\rm (A3)}$, and \eqref{leq1} that $(R-S)^{-b} Y(0)$ has a bound on $I \times [0,T]$ independent of $T$.

We now turn to the estimate of $(R-S)^{-b} J_1$. With the help of Lemma \ref{estimate}, \eqref{M}, and \eqref{leq1}, there exists a constant $C_1>0$ depending only on $b$, $\nu$, and $M$ such that
\begin{equation}
\label{I103}
|(R-S)^{-b}(x,t)J_1(t)|\le C_1\int_{0}^{t} k^2(x(s)) S(x(s),s) \, ds\quad \text{for all} \ (x,t) \in I\times [0,T]
\end{equation}
as $\nu$ is sufficiently small. To control the integral in \eqref{I103}, we note that $dx/dt = \lambda_{1}$ and $S \leq \lambda_{1}$, which yield that, for all $s \in [0,t]$,
\begin{equation*}
\label{intS1}
x_0+\int_{0}^{s}  S (x(\tau),\tau)  \, d\tau \leq x(s).
\end{equation*}
Applying \eqref{k1}, we get that $k$ is decreasing. Thus,
\begin{equation*}
k(x(s)) \leq k \bigg(x_0+\int_{0}^{s}  S (x(\tau),\tau) \, d\tau \bigg) \quad \text{for}\ s \in [0, t].
\end{equation*}
We conclude from change of variables and condition ${\rm (A1)}$ that
\begin{equation}\label{intS3}
\int_{0}^{t} k^2(x(s)) S(x(s),s) \, ds \leq \| k \|_{C^0} \int_{0}^{t} k \bigg(x_0+\int_{0}^{s}  S (x(\tau),\tau) \, d\tau \bigg) S(x(s),s) \, ds \le \| k \|_{C^0} \int_{x_B}^{\infty} k(y) dy.
\end{equation}
As a consequence of Remark \ref{remark1}, \eqref{I103}, and \eqref{intS3}, we obtain that $(R-S)^{-b} J_1$ has a bound on $I \times [0,T]$ independent of $T$.

Last, similar to \eqref{I103}, we also have
\begin{equation}
\label{I104}
|(R-S)^{-b}(x,t)J_2(t)|\le -C_2\int_{0}^{t} k'(x(s)) S(x(s),s) \, ds:=J_3(t)\quad \text{for all} \ (x,t) \in I\times [0,T]
\end{equation}
as $\nu$ is sufficiently small, where $C_2>0$ depends only on $b$, $\nu$, and $M$ and we have used the fact that $k' \leq 0$. Since $dx/dt = \lambda_1$ and $S \leq \lambda_1$,
\begin{equation}
\label{J3}
J_{3}(t) =  -C_2 \int_{0}^{t} \frac{S(x(s),s)}{\lambda_{1} (x(s),s)} \ k'(x(s)) x'(s) \, ds\le -C_2 \int_{0}^{t} k'(x(s)) x'(s) \, ds  \le -C_2 \int_{x_B}^{\infty} k'(z)  \, dz.
\end{equation}
Remark \ref{remark1} indicates that
$$- \int_{x_B}^{\infty} k'(z)  \, dz = k(x_B),$$
which together with \eqref{I104}--\eqref{J3} proves that $(R-S)^{-b} J_2$ has a bound on $I \times [0,T]$ independent of $T$.
We complete the proof.
\end{proof}

To establish the global existence of the classical solution $[\rho_{\eta}, v]^T$ for \eqref{a.1}, we make use of \eqref{a.5} and \eqref{rhov} to obtain that
\begin{equation}
\label{sol}
\rho_{\eta} = \left(\frac{ (\gamma -1)^2 }{16 \gamma} \right)^{\frac{1}{\gamma -1}} \left(R-S \right)^{\frac{2}{\gamma-1}} \ \ \text{and} \ \ v = \frac{S+R}{2}.
\end{equation}
Since Theorem \ref{GC} says that $S$ and $R$ in \eqref{RIE} are $C^1$ functions on $I \times [0, \infty)$, it is easy to see that $\rho_{\eta}$ and $v$ in \eqref{sol} are also $C^1$ functions on $I \times [0, \infty)$ under the assumptions of Theorem \ref{GC}. That is, we have the following corollary.

\begin{cor}\label{cor}
Under the assumptions of Theorem \ref{GC}, problem \eqref{a.1} admits a unique global $C^1$ solution for sufficiently small $\eta$.
\end{cor}

In Lemmas \ref{key3} and \ref{key4}, we see that \eqref{k2} and \eqref{k3} are simplified but more strict forms than \eqref{S0}--\eqref{R0} and \eqref{SB}--\eqref{RB}, respectively. It is easy to find that there is a large class of $S_0(x)$, $R_0(x)$, $S_B(t)$, and $R_B(t)$ that satisfy all conditions of ${\rm (A2)}$--${\rm (A5)}$, \eqref{k2}, and \eqref{k3}.
Hence there exist a lot of functions that satisfy all the initial and boundary conditions of Theorem \ref{GC}.

It is worth noting that, in the proof of Theorem \ref{GC}, we have $Y(t) \geq 0$ along $\Gamma^{1}$ and $Z(t) \geq 0$ along $\Gamma^{2}$ for any given point $(x,t) \in I \times [0, \infty)$, where $\Gamma^{i}$, $i=1,2$, is the $i$th characteristic curve passing through $(x,t)$. Thus, the $C^1$ solution in Theorem \ref{GC} preserves inequalities in \eqref{S0}--\eqref{R0} and \eqref{SB}--\eqref{RB} for any $x \geq x_B$ and $t \geq 0$. More precisely, for $ b \neq -1$ and  $(x,t) \in I \times [0, \infty)$, we have
\begin{equation*}
\label{SRx}
\begin{aligned}
& \frac{k(x)S(x,t)}{2} \left( \frac{1}{-b} - \frac{1}{b+1} \xi(x,t) \right)  \leq S_x(x,t), \\
& \frac{k(x)S(x,t)}{2} \left( \frac{1}{-b} - \frac{1}{b(b+1)} \xi(x,t) \right) \leq R_{x}(x,t), \\
\end{aligned}
\end{equation*}
and
\begin{equation*}
\label{SRt}
\begin{aligned}
& S_t (x,t) \leq \frac{k(x) S^2(x,t)}{2} \left(  \frac{1}{b} +  \frac{1}{b+1}\xi(x,t) + \frac{1-b}{2(b+1)(1-2b)}\xi^2(x,t) \right),  \\
& R_t (x,t) \leq \frac{k(x) S^2(x,t)}{2} \left(  \frac{1}{b} +  \frac{b+2}{b(b+1)} \xi(x,t) - \frac{b^2+3b-2}{2b(b+1)(1-2b)} \xi^2(x,t) \right).
\end{aligned}
\end{equation*}
The corresponding inequalities for $b =-1$ also hold. We omit them for brevity. In particular, as $\nu$ is sufficiently small, we have
\begin{equation}
\label{SRxt}
\frac{k(x) S(x,t)}{-4b} \leq  \min \Big\{ S_x (x,t), R_x(x,t) \Big\} \ \ \text{and} \ \ \max \Big\{ S_t(x,t), R_t(x,t) \Big \} \leq \frac{k(x) S^2(x,t)}{4b}
\end{equation}
on $I \times [0, \infty)$ for all $b <0$. Making use of \eqref{SRxt}, we are able to describe the behavior of $S$, $R$, and related physical quantities as $t \rightarrow \infty$ in the following two propositions.

\begin{prop}\label{prop1}
Under the assumptions of Theorem \ref{GC}, for any fixed $x \in I$ with $k(x) >0$ and $\nu$ sufficiently small, $S(x,t)$, $R(x,t)$, $v(x,t)$ decrease to zero, and $\rho (x,t)$ tends to zero as $t\rightarrow\infty$. In this case, the convergent rates of $S(x,t), R(x,t), v(x,t)$, and $\rho^{(\gamma-1)/2}(x,t)$ are $O(1/t)$ as $t\rightarrow\infty$. Furthermore, all the convergences are uniform on $[x_B, x_*]$ for any $x_{*} \in I$ with $k(x_{*}) > 0$.
\end{prop}

\begin{proof}
Let $x \in I$ be fixed. Since $S >0$ by Lemma \ref{2.1}, the second inequality in \eqref{SRxt} implies $S_t(x,t)/S^2(x,t) \leq k(x)/(4b) $. Integrating the both sides yields that
\begin{equation}
	\label{Stend01}
	S(x,t) \leq  \left(\frac{1}{S_0(x)} - \frac{k(x)}{4b} \ t \right)^{-1}
\end{equation}
and hence
\begin{equation}
	\label{Stend0}
	\lim_{t \rightarrow \infty} S(x,t) = 0\quad \text{for}\ \nu\ \text{sufficiently small}.
\end{equation}
It is obvious from the definition of $\xi$ and Lemma \ref{estimate} that
\begin{equation*}
\label{SRS}
S \leq R = S (\xi +1 )\leq 2S\quad \text{for}\ \nu\ \text{sufficiently small},
\end{equation*}
which along with \eqref{Stend0} immediately gives that
\begin{equation}
\label{Rtend0}
\lim_{t \rightarrow \infty} R(x,t) = 0\quad \text{for}\ \nu\ \text{sufficiently small}.
\end{equation}
Therefore, for $\nu$ sufficiently small, $v(x,t)$ and $\rho(x,t)$ tend to zero as $t\rightarrow\infty$ by \eqref{Stend0}--\eqref{Rtend0} and the fact that $v=(S+R)/2$ and $R - S = O(\sqrt{\nu}) \rho^{(\gamma-1)/2}$. The decreasing properties of $S$, $R$ and $v$ follow directly from $S_t \leq 0$ and $R_t \leq0$ shown in \eqref{SRxt}. In addition, the above arguments also show that the convergence rates of $S(x,t), R(x,t), v(x,t)$, and $\rho^{(\gamma-1)/2}(x,t)$ are $O(1/t)$ as $t\rightarrow\infty$.

Last, for any $x_{*} \in I$ such that $k(x_{*}) > 0$, we have $k(x_{*}) \leq k(x)$ for all $x \in [x_B, x_{*}]$ since $k$ is decreasing by inequality \eqref{k1}. Consequently, the convergence of $S$ is uniform on $[x_B, x_*]$ by \eqref{Stend01}.
This together with the preceding arguments shows that all the convergences are uniform on $[x_B, x_*]$.
\end{proof}

\begin{prop}\label{prop2}
Under the assumptions of Theorem \ref{GC}, $S(x,t)$ is increasing, and $R(x,t)$ and $\xi(x,t)$ are decreasing along the characteristic curves $\Gamma^{i}$ $(i=1,2)$ as $\nu$ is sufficiently small. In this case, $\xi \rightarrow 0$ along $\Gamma^{i}$ $(i=1,2)$ if, and only if, $\rho \rightarrow 0$ along $\Gamma^{i}$ $(i=1,2)$.
\end{prop}

\begin{proof}
In the proof of Theorem \ref{GC}, we have already showed that $S$ is increasing and $R$ is decreasing along both $\Gamma^{1}$ and $\Gamma^{2}$.
By the definition of $\xi$, it is obvious that $\xi(x,t)$ is also decreasing along both $\Gamma^{1}$ and $\Gamma^{2}$.
Moreover, if $\Gamma^{i}$ passing through $(x_i(t),t)$ emanates from $(x_i(t_*),t_*)$ for $i=1,2$, then
\begin{equation*}
4 \sqrt{\frac{\nu}{\gamma-1}} \frac{\rho^{\frac{\gamma-1}{2}} (x_i(t),t)}{M} \leq
\xi(x_i(t),t) \leq
4 \sqrt{\frac{\nu}{\gamma-1}} \frac{\rho^{\frac{\gamma-1}{2}} (x_i(t),t)}{S(x_i(t_*), t_*)},
\end{equation*}
where $M$ is given in Lemma \ref{2.1}, which shows the equivalence of $\xi \rightarrow 0$ and $\rho \rightarrow 0$ along $\Gamma^{i}$.
The proof is finished.
\end{proof}

\subsection{Spherical symmetric supersonic flows of the compressible Euler equations on bounded domains.}

As an application of Theorem \ref{GC}, we consider the initial-boundary value problem \eqref{a.8} for $N$-dimensional spherical symmetric supersonic flow of the compressible Euler equations, where $a(x) = x^{N-1}$. By the definition of $k(x)$ given in \eqref{k}, we have
\begin{equation}
\label{k(x)}
k(x) = \frac{N-1}{x}.
\end{equation}
Similarly, we rewrite this problem into \eqref{RIE} in terms of the Riemann invariants $R$ and $S$ as described in Subsection 2.1.
In the following we show how Theorem \ref{GC} can be used to obtain the global existence of the $C^1$ solution for \eqref{RIE}.

We let $x_B >0$ and recall $I = [x_B, x_C]$ or $[x_B, \infty)$. Then it is clear from \eqref{k(x)} that $k(x) \in C^1(I)$ has a bounded $C^1$ norm. Under the assumption $1<\gamma<3$, i.e., $b<0$, where $b$ is defined in \eqref{h}, it is of interest to note that inequality \eqref{k1} is equivalent to
$$
1<\gamma<1 + \frac{2}{N} \ \ \Big(\text{i.e.,} \ \  b < -\frac{N-1}{2} \Big).
$$
This coincides with the adiabatic index for the polyatomic ideal gas. Next, if $S_0(x)$ and $R_0(x)$ satisfy condition ${\rm (A3)}$ and the following
\begin{equation}
\label{initial1}
\frac{(N-1)S_0(x)}{-bx} \leq S'_0(x) \ \ \text{and} \ \ \frac{(N-1)R_0(x)}{-bx} \leq R'_0(x)
\end{equation}
for all $x \in I$, then as $\nu$ is sufficiently small, $S_0(x)$ and $R_0(x)$ also satisfy condition ${\rm (A3)}$, \eqref{S0}, and \eqref{R0} for all $x \in I$. By inequalities in \eqref{initial1}, we see that for all $x \in I$
\begin{equation}
\label{initial2}
x^{\frac{N-1}{b}}_B S_0(x_B) \ x^{\frac{N-1}{-b}} \leq S_0(x) \ \ \text{and} \ \ x^{\frac{N-1}{b}}_B R_0(x_B) \ x^{\frac{N-1}{-b}} \leq R_0(x).
\end{equation}
Since $S_0(x)$ and $R_0(x)$ have bounded $C^0$ norms, inequalities in \eqref{initial2} force us to consider a finite interval $I = [x_B, x_C]$ for some $x_C < \infty$. On the other hand, using condition ${\rm (A5)}$ and Lemma \ref{2.1}, we have
\begin{equation}
\label{initialR}
v(x,t) \leq R_{0}(x_C)
\end{equation}
for all $(x,t) \in [x_B, x_C] \times [0, \infty)$. In physics, no speed cannot exceed the light speed $c$. Hence we suppose further that
\begin{equation}
\label{<c}
R_{0}(x_C)<c.
\end{equation}
Then $v(x,t) < c$ for all $(x,t) \in [x_B, x_C] \times [0, \infty)$ by \eqref{initialR} and \eqref{<c}. With the preceding arguments and Theorem \ref{GC} in mind, we then have the following result.

\begin{thm}\label{GC2}
Consider \eqref{RIE} with $1 < \gamma < 1 + 2/N$, $k(x) = (N-1)/x$, and  $I=[x_B, x_C]$, where $0<x_B<x_C$, and $x_C$ satisfies the inequalities in \eqref{initial2} and \eqref{<c}. Suppose that $S_0(x)$, $R_0(x)$, $S_B(t)$, $R_B(t)$ satisfy conditions ${\rm (A1)}$--${\rm (A5)}$, \eqref{SB}, and \eqref{RB}. Then \eqref{RIE} admits a unique global $C^1$ solution on $[x_B, x_C] \times [0, \infty)$ for sufficiently small $\nu$.
\end{thm}

\section{Numerical experiments}
\setcounter{section}{4}

In this section, we present several numerical experiments to validate our theoretical results. In the following experiments, we set the domain $I$ to be $[1, 10]$ so that $x_B =1$, and we choose $\gamma=7/5$, i.e., $b =-2$ via \eqref{h}. All the numerical results are computed by the 5th-order WENO scheme \cite{JS1996} with the 3rd-order TVD Runge-Kutta method in time to solve the initial-boundary value problem \eqref{a.8}. \\

\noindent {\bf Experiment 1. $a(x)=2-(1/x)$}. \\

In the first numerical experiment, we consider the case $a(x)=2-(1/x)$, i.e., $k(x) = 1/(2x^2 -x)$ by recalling $k(x) = a'(x)/a(x)$. Note that $k(x_B)=a(x_B)=1$ since $x_B =1$. By Remark \ref{remark1}, we know that the function $k$ satisfies condition $\eqref{k1}$. For the boundary data (BD), we set
\begin{equation}
\label{BD}
\text{BD}: \ \ \
S_{B}(t) = \frac{1 - \sqrt{\nu}}{1 + t} \ \ \text{and} \ \  R_{B}(t) = \frac{1}{1+t}.
\end{equation}
It is easy to check that the boundary data satisfy $\eqref{k3}$. For the initial data (ID), we choose $s_0=S_B(0)=1-\sqrt{\nu}$ and $r_0=R_B(0)=1$, and set $s_0'$ and $r_0'$ by condition ${\rm (A2)}$ as follows:
\begin{equation*}
s'_0 = \frac{\frac{\gamma -1}{8} k(x_B)( r_0^2 - s_0^2)-S'_B(0)}{\frac{\gamma +1}{4} s_0  + \frac{3-\gamma}{4} r_0} = \frac{20 - 18 \sqrt{\nu} - \nu}{20 - 12 \sqrt{\nu} }
\end{equation*}
and
\begin{equation*}
r'_0 = \frac{\frac{1-\gamma }{8} k(x_B) ( r_0^2 - s_0^2)-R'_B(0)}{\frac{\gamma +1}{4} r_0  + \frac{3-\gamma}{4} s_0} = \frac{20 - 2 \sqrt{\nu} + \nu}{20 - 8 \sqrt{\nu} }.
\end{equation*}
Then we define
\begin{equation}
\label{ID}
\text{ID}: \ \ \
S_{0}(x) = s_0 a(x)^{\frac{s_0'}{s_0}} \ \ \text{and} \ \  R_{0}(x) = r_0 a(x)^{\frac{r_0'}{r_0}}.
\end{equation}
In this case, we have
\begin{equation*}
\frac{S_0(x)k(x)}{-b} \leq S'_0(x) = S_0(x)k(x)\frac{s'_0}{s_0} \ \ \text{and} \ \ \frac{R_0(x)k(x)}{-b} \leq R'_0(x) = R_0(x)k(x)\frac{r'_0}{r_0}.
\end{equation*}
Hence condition $\eqref{k2}$ holds. In addition, from the above settings, it is a routine matter to check that conditions ${\rm (A1)}$--${\rm (A5)}$ are satisfied. In particular, conditions ${\rm (A3)}$ and ${\rm (A4)}$ show that the initial and boundary data of the density $\rho$ and the velocity $v$ obtained by \eqref{rhov} are $C^1$-bounded whenever $\nu \leq 0.1$. We solve \eqref{a.8} with $\nu=0.1, 0.001, 0.00001$ until $t=10$. Space grid size is set to be $\Delta x=0.01$ and time step is set to be $\Delta t = 0.1 \Delta x$.  We monitor $\|\rho(\cdot,t)\|_\infty$, $\|v(\cdot,t)\|_\infty$, $\|\rho_x(\cdot,t)\|_\infty$, and $\|v_x(\cdot,t)\|_\infty$ and plot them in Figure \ref{Fig.1}. The $C^1$ norms of the solutions are bounded at all times as we have shown in Theorem \ref{GC}. We also observe that the solutions seem to be convergent as $\nu$ tends to $0$. This coincides with the results in \cite{Lee}.

\begin{figure}[h]
\centering
\includegraphics[scale=0.9]{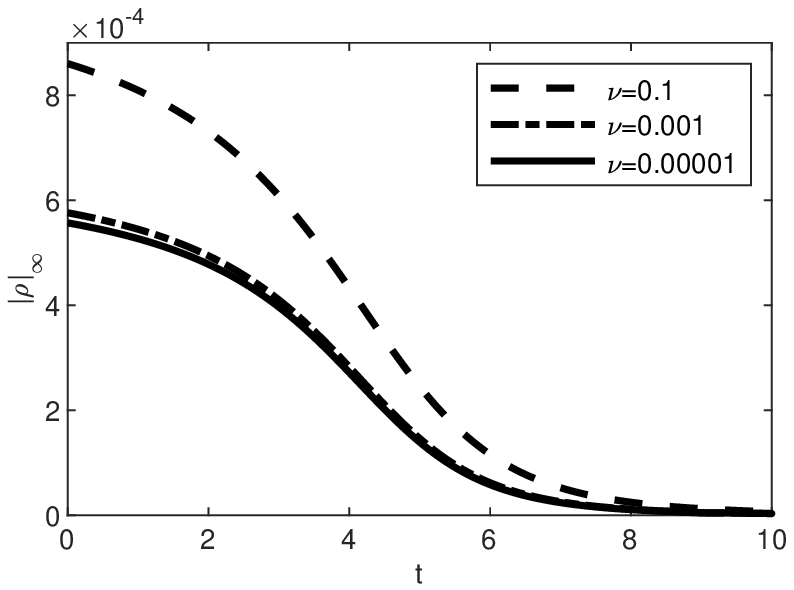}
\includegraphics[scale=0.9]{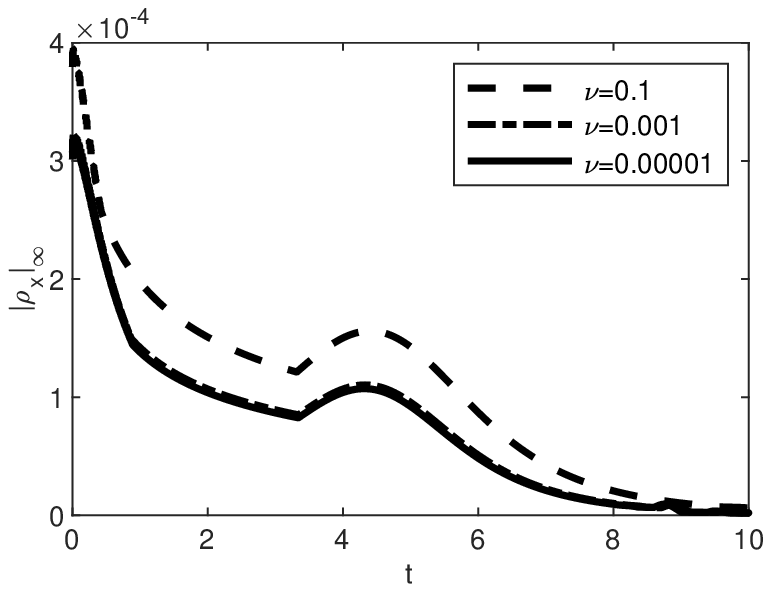}\\
\includegraphics[scale=0.9]{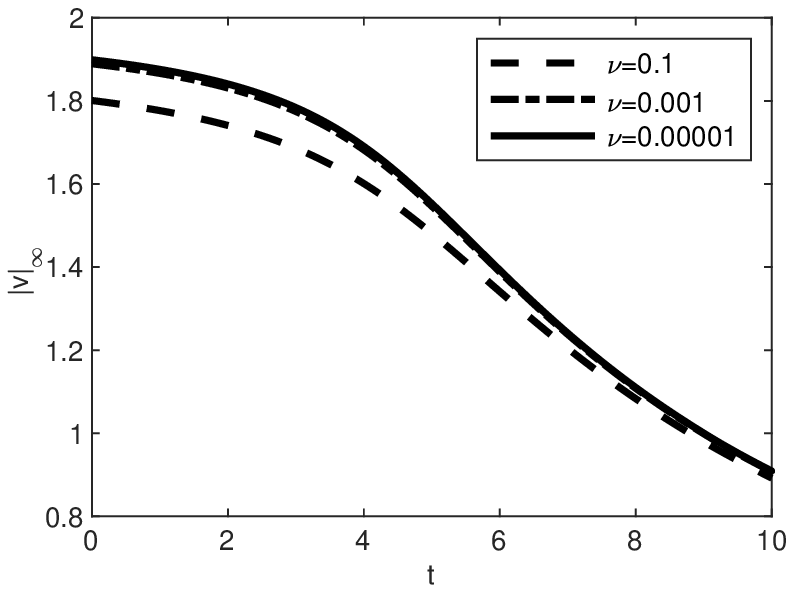}
\includegraphics[scale=0.9]{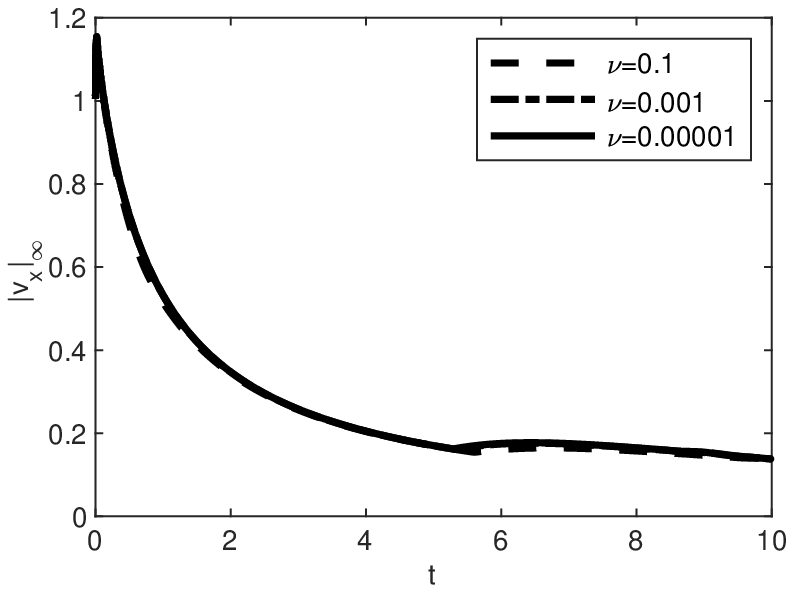}
\caption{Numerical Experiment 1 with $a(x)=2-(1/x)$. $\|\rho(\cdot,t)\|_\infty$, $\|v(\cdot,t)\|_\infty$, $\|\rho_x(\cdot,t)\|_\infty$, and $\|v_x(\cdot,t)\|_\infty$  from $t=0$ to $t=10$ for $\nu=0.1, 0.001, 0.00001$.}\label{Fig.1}
\end{figure}

\noindent {\bf Experiment 2. $a(x)=e^{1-(1/x)}$}. \\

In the second numerical experiment, we consider the case $a(x)=e^{1-(1/x)}$, i.e., $k(x) = 1/x^2$.
Since $x_B =1$, we have $k(x_B)=a(x_B)=1$. In addition, we have seen from Remark \ref{remark1} that the function $k$ satisfies the condition $\eqref{k1}$.
For the initial and boundary data, we use the same setting as in Experiment 1. That is, \eqref{BD} and \eqref{ID} are given. As a consequence, it can be seen from the above that all the conditions required by Theorem \ref{GC} are satisfied. We solve \eqref{a.8} with $\nu=0.1, 0.001, 0.00001$ until $t=10$ and $\Delta x=0.01$, $\Delta t = 0.1 \Delta x$.  We monitor $\|\rho(\cdot,t)\|_\infty$, $\|v(\cdot,t)\|_\infty$, $\|\rho_x(\cdot,t)\|_\infty$ and $\|v_x(\cdot,t)\|_\infty$, and plot them in Figure \ref{Fig.2}. Again, $C^1$ norms of the solutions are bounded at all times, and the solutions seem to be convergent as $\nu$ tends to 0.

\begin{figure}[h]
\centering
\includegraphics[scale=0.9]{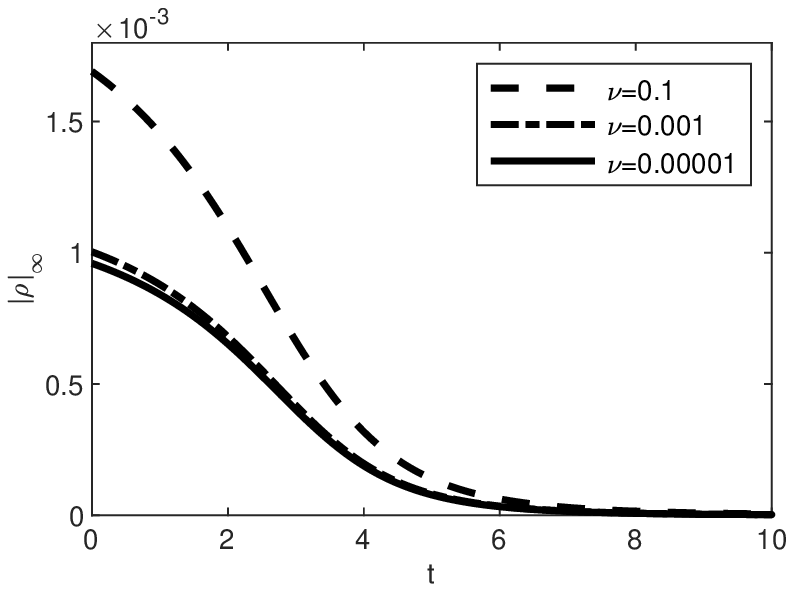}
\includegraphics[scale=0.9]{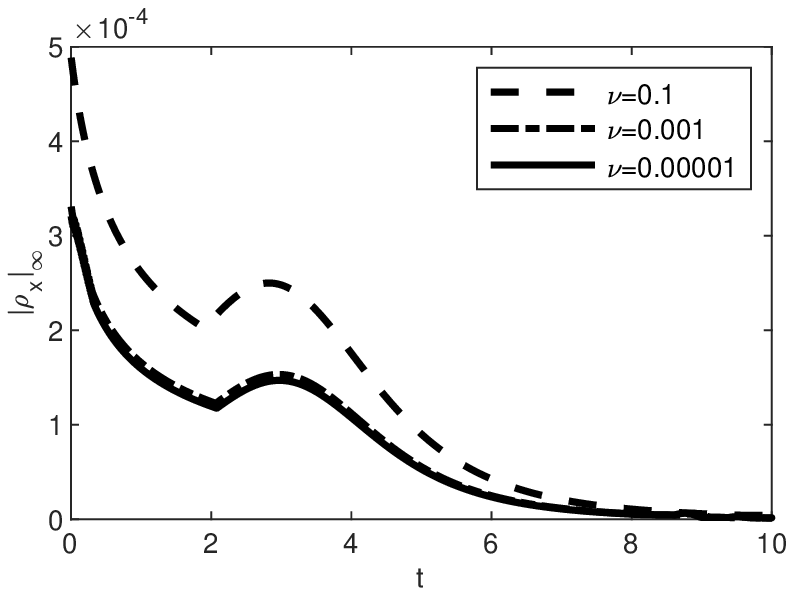}\\
\includegraphics[scale=0.9]{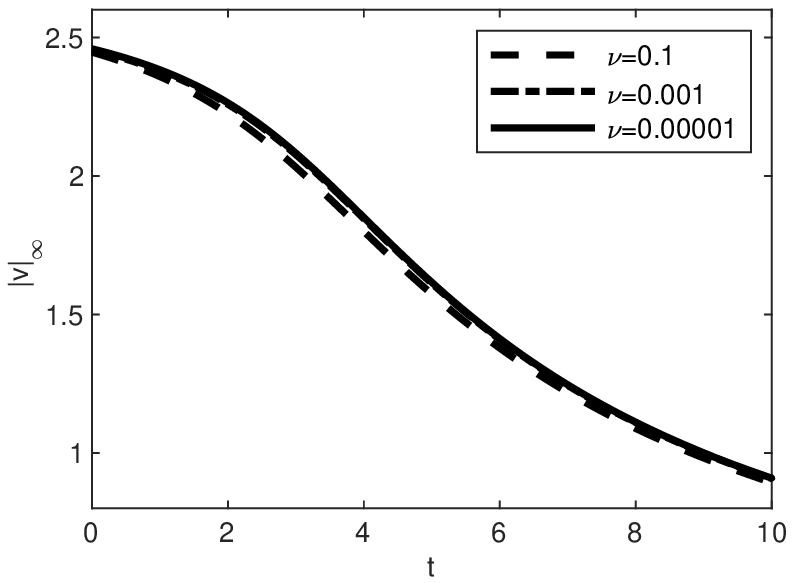}
\includegraphics[scale=0.9]{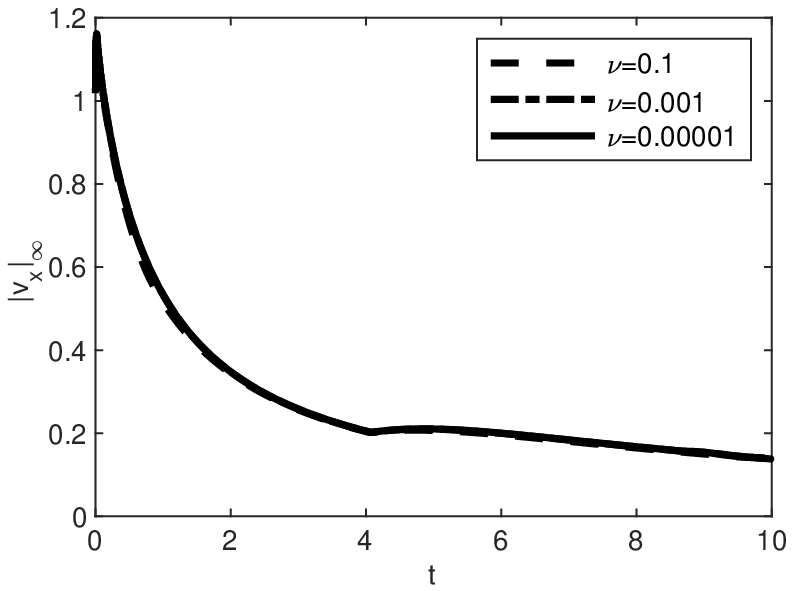}
\caption{Numerical Experiment 2 with $a(x)=e^{1-(1/x)}$. $\|\rho(\cdot,t)\|_\infty$, $\|v(\cdot,t)\|_\infty$, $\|\rho_x(\cdot,t)\|_\infty$, and $\|v_x(\cdot,t)\|_\infty$  from $t=0$ to $t=10$ for $\nu=0.1, 0.001, 0.00001$.}\label{Fig.2}
\end{figure}

\section*{Acknowledgments}

The work of  Ying-Chieh Lin is supported in part by Ministry of Science and Technology of Taiwan under research grant MOST-109-2115-M-390-001-MY2. Jay Chu is partially supported by Ministry of Science and Technology of Taiwan under research grant MOST-109-2115-M-007-016-. John M. Hong is partially supported by Ministry of Science and Technology of Taiwan under research grant MOST-108-2115-M-008-005-MY2. Hsin-Yi Lee is partially supported by the Ministry of Science and Technology of Taiwan under research Grant MOST-109-2115-M-008-013-MY2. The authors thank Prof. Hailiang Liu at the Department of Mathematics, Iowa State University for introducing the generalized Lax transformations to them.


\begin{thebibliography}{99}

\bibitem{Bhatnagar} M. Bhatnagar and H. Liu, Critical thresholds in one dimensional damped Euler-Poisson systems, Math. Mod. Meth. Appl. Sci., 30(05):891-916, 2020.

\bibitem{Chu} J. Chu, J. M. Hong, and H.-Y. Lee, Approximation and existence of vacuum states in the multiscale flows of compressible Euler equations, Multiscale Model. Simul., 18 (1) (2020), pp. 104-130.

\bibitem{Chen1} G. Chen, R. Pan, and S. Zhu, Sigularity formation for the compressible Euler equations, SIAM J. Math. Anal., 49 (4) (2017), pp. 2591-2614.

\bibitem{Chen4} G. Chen and R. Young, Smooth solutions and singularity formation for the inhomogeneous nonlinear wave equation, J. Differential Equations, 252 (3) (2012), pp. 2580-2595.

\bibitem{Chen5} G. Chen, R. Young, and Q. Zhang, Shock formation in the compressible Euler equations and related systems, J. Hyper. Differential Equations, 10 (1) (2013), pp. 149-172.

\bibitem{GChen2} G.-Q. Chen and M. Feldman, Existence and stability of multidimensional transonic flows through an infinite nozzle of arbitrary cross-sections, Arch. Rational Mech. Anal., 184 (2007), pp. 185-242.

\bibitem{Engelberg} S. Engelberg, H. Liu and E. Tadmor, Critical thresholds in Euler-Poisson equations, Indiana Univ. Math. J., 50 (2001), pp. 109-157.

\bibitem{Glimm1} J. Glimm, Solutions in the large for nonlinear hyperbolic systems of equations. Comm. Pure Appl. Math., 18 (1965), pp. 697-715.

\bibitem{Hsiao} L. Hsiao and T. T. Li, Global smooth solution of Cauchy problems for a class of quasilinear hyperbolic systems, Chinese Ann. Math. 4B (1983), pp. 109-115.

\bibitem{Hong} J. M. Hong, An extension of Glimm's method to inhomogeneous strictly hyperbolic systems of conservation laws by ``weaker than weak" solutions of the Riemann problem, J. Differential Equations, 222 (2006), pp. 515-549.

\bibitem{Jiang}  M. Jiang, D. Li, and L. Ruan, Existence of a global smooth solution to the iniital boundary-value problem or the $p$-system with nonlinear damping, Proc. Roy. Soc. Edinburgh Sect A, 143 (2013), pp. 1243-1253.

\bibitem{JS1996}
G.-S. Jiang and C.-W. Shu.
\newblock Efficient implementation of weighted {ENO} schemes.
\newblock {\em J. Comput. Phys.}, 126(1):202--228, 1996.

\bibitem{Lax} P. D. Lax, Development of sigularities of solutions of nonlinear hyperbolic partial differential eqautions, J. Math. Phys. 5 (1964), pp. 611-613.

\bibitem{Lee} H.-Y. Lee, J. Chu, J. M. Hong, and Y.-C. Lin, $L^1$ convergences and convergence rates of approximate solutions for
 compressible Euler equations near vacuum, Reserach in the Mathematical Sciences, 7:6 (2020).

\bibitem{TLi1} T. T. Li, Global Solutions for Quasilineaar Hyperbolic Systems, Wiley, New York, 1994.

\bibitem{TLi2} Li Tatsien and Qin Tiehu, Global smooth solutions for a class of quasiliniear hperbolic systems with dissipative terms, Chinese Ann. Math. Ser. B 6 (1985), pp. 199-210.

\bibitem{TLi3} T. T. Li and W. C. Yu, Boundary Value Problems for Quasilinear Hyperbolic Systems, Duke University Mathematics Series, V. Durham, NC 27706, Duke University, Mathematics Department. X, 1985.

\bibitem{Li1} T. Li and H. L. Liu, Critical thresholds in hyperbolic relaxation systems, J. Differential Equations, 247 (2009), pp. 33-48.

\bibitem{Lin2} L. Lin, H. Liu and T. Yang, Existence of globally bounded continuous solutions for nonisentropic gas dynamics equations, J. Math. Anal. Appl., 209, (1997), pp. 492-506.

\bibitem{TPLiu3} T.-P. Liu, Quasilinear hyperbolic systems, Commun. Math. Phys., 68 (1979), pp. 141-172.

\bibitem{TPLiu4} T.-P. Liu, and J. A. Smoller, On the vacuum state for the isentropic gas dynamics equations, Adv. Appl. Math., 1 (1980), pp. 345-359.

\bibitem{TPLiu5} T.-P. Liu and T. Yang, Compressible flow with vacuum and physical sigularity, Methods Appl. Anal., 7 (3) (2000), pp. 495-510.

\bibitem{Luo1} T. Luo and J. Smoller, Existence and non-linear stability of rotating star solutions of the compressible Euelr-Poisson Equations, Arch. Rational Mech. Anal., 191 (2009), pp. 447-496.

\bibitem{Luo2} T. Luo, Z. Xin, and H. Zeng, Well-posedness for the motion of physical vacuum of the three-dimensional compreessible Euler equations with or without self-gravitation,  Arch. Rational Mech. Anal., 213 (2014), pp. 763-831.

\bibitem{Makino} T. Makino and S. Takeno, Initial boundary value problem for the spherically symmetric Motion of isentropic gas, Japan J. Indust. Appl. Math., 11 (1994), pp. 171-183.

\bibitem{Temple} B. Temple, Global solution of the Cauchy problem for a class of $2 \times 2$ non-strictly hyperbolic conservation laws, Adv. Appl. Math. 3 (1982), pp. 335-375.

\bibitem{Tsuge3} N. Tsuge, isentropic gas flow for the compressible Euler equation in a nozzle, Arch. Rational Mech. Anal., 209 (2013), pp. 365-400.

\bibitem{Wang} Wang Jiangua and Li Caizhong, Global regulairty and formation of sigularirties of solution for quasilinear hyperbolic systems with dissipation, Chinese Ann. Math. Ser. A 9 (1988), pp. 509-523.

\bibitem{Wang2} C. Wang and Z. Xin, Smooth Transonic Flows of Meyer Type in De Laval Nozzles, Arch. Rational Mech. Anal., 232 (2019), pp. 1597-1647.

\bibitem{Yamaguti} M. Yamaguti and T. Nishida, On some global solution for the quasilinear hyperbolic equations, Funkcial. Ekvac. 11 (1968), pp. 51-57.

\bibitem{Yang1}  T. Yang, and C. Zhu, Existence and non-existence of global smooth solutions for $p$-system with relaxation, J. Diff. Eqns. 161 (2000), pp. 321-336.

\bibitem{Yang2}  T. Yang, C. Zhu, and H. Zhao, Global smooth solutions for a class of quasilinear hyperbolic systems with dissipative terms, Proc. Roy. Soc. Edinburgh Sect A, 127 (1997), pp. 1311-1324.

\bibitem{Ying} L. Ying, T. Yang, and C. Zhu, Existence of global smooth solutions for euler equations with symmetry, Commun. Patial Differential Equations, 22 (7 $\&$ 8) (1997), pp. 1361-1387.

\bibitem{Zhu} C. J. Zhu, Global resolvability for a viscoelastic model with relaxation, Proc. Roy. Soc. Edinburgh Sect A, 125 (1995), pp. 1277-1285.


\end{thebibliography}
\end{document}